\documentclass[11pt]{amsart}
\usepackage{amsfonts,amssymb,amsmath}
\usepackage[active]{srcltx}
\usepackage{amsmath}
\usepackage{color}

\newtheorem{theorem}{\bf Theorem}[section]

\newtheorem{remark}{\bf Remark}[section]
\newtheorem{lemma}{\bf Lemma}[section]
\newtheorem{definitions}{\bf Definition}[section]
\numberwithin{equation}{section}

\title[Risk Sensitive Games for Continuous Time MDP]
{Nonzero-Sum Risk Sensitive Stochastic Games for Continuous Time Markov Chains}
{\author{ Mrinal K. Ghosh, K. Suresh Kumar and Chandan Pal }
\address{Department of
Mathematics, Indian Institute of Science, Bangalore -560012, India. }
\address{Department of
Mathematics, Indian Institute of Technology Bombay, Mumbai -
400076, India. }}

\email{mkg@math.iisc.ernet.in, suresh@math.iitb.ac.in, chandan14@math.iisc.ernet.in}
\begin{document}
\maketitle

\begin{abstract}
\noindent We study nonzero-sum stochastic games for continuous time Markov chains on a denumerable state space with risk sensitive discounted and ergodic cost criteria. For the discounted cost criterion we first show that the corresponding system of coupled HJB equations has an appropriate solution. Then under an additional additive structure on the transition rate matrix and payoff functions, we establish the existence of a Nash equilibrium in Markov strategies. For the ergodic cost criterion we assume  a Lyapunov type stability assumption and a small cost condition. Under these assumptions we show that the corresponding system of coupled HJB equations admits a solution which leads to the existence of Nash equilibrium in stationary strategies.
\end{abstract}

\vspace{10mm}


\vspace{5mm}
%\noindent
%{\it 2000 Mathematics Subject Classification.}  Primary 93E20, Secondary 49L20, 60J27. 

%\begin{AMS}
 %91A15, 91A25
%\end{AMS}

\pagestyle{myheadings}
\thispagestyle{plain}
\markboth{MRINAL K. GHOSH, K SURESH KUMAR 
 AND CHANDAN PAL }{Risk Sensitive Games for Continuous Time MDP}

\section{Introduction}

We study nonzero-sum stochastic games on  infinite time horizon for continuous time Markov chains  on a denumerable state space. The performance evaluation criterion is exponential of integral cost which addresses the decision makers (i.e., players) attitude towards risk. In other words we address the problem of nonzero-sum risk sensitive stochastic games involving continuous time Markov chains. In the literature of stochastic games involving continuous time Markov chains, one usually studies the integral of the cost (see, e.g., Guo and Hern$\acute{\rm a}$ndez-Lerma  \cite{GH1}, \cite{GH2}, \cite{GH3}) which is the so called risk-neutral situation. In the exponential of integral cost, the evaluation criterion is multiplicative as opposed to the additive nature of evaluation criterion in the integral of cost case. This difference makes the risk sensitive case significantly different from its risk neutral counterpart. The study of risk sensitive criterion was first introduced by Bellman in \cite{RB}; see Whittle \cite{PW} and the references therein. Though this criterion  is studied extensively in the context of stochastic dynamic optimization both in  discrete and continuous time (see, e.g., Cavazos-Cadena and Fernandez-Gaucherand \cite{CF}, Di Masi and Stettner \cite{MS1}, \cite{MS2}, \cite{MS3}, Fleming and Hern$\acute{\rm a}$ndez-Hern$\acute{\rm a}$ndez \cite{FH}, \cite{FH1}, Fleming and McEneaney \cite{FM}, Hern$\acute{\rm a}$ndez-Hern$\acute{\rm a}$ndez and Marcus \cite{HM1}, \cite{HM2}, Howard and Matheson \cite{HM3}, Jacobson \cite{J},   Rothblum \cite{R}), the corresponding results for stochastic dynamic games are rather sparse. Notable exceptions are Basar \cite{TB}, El-Karoui and Hamadene \cite{KH}, James et al. \cite{JBE}, Klompstra \cite{KL}.
Recently risk sensitive continuous time Markov decision processes has been studied by Ghosh and Saha \cite{GhoshSaha}, Kumar and Pal \cite{SP}, \cite{SP1}. In this paper we extend the results of the above three papers to the nonzero-sum stochastic games. In particular we establish the existence of a Nash equilibria for risk-sensitive discounted and long-run average (or ergodic) cost criteria. 

The rest of this paper is organized as follows: Section 2 deals with the problem description and preliminaries. The discounted cost criterion is analyzed in Section 3. Here we first establish the existence of a solution to the corresponding coupled Hamilton-Jacobi- Bellman (HJB) equations. Then under certain additive structure on the transition rate (infinite) matrix and payoff functions, we establish the existence of a Nash equilibrium in Markov strategies. In Section 4, we turn our attention to the ergodic cost criterion. Under a Lyapunov type stability assumption and a small cost assumption, we carry out the vanishing discount asymptotics. This leads to the existence of appropriate solutions to the coupled HJB equations for the ergodic cost criterion. This in turn leads to the existence of a Nash equilibrium in stationary strategies. We conclude our paper in Section 5 with a few remarks.

\section{Problem Description and Preliminaries}
For the sake of notational simplicity we treat two player game. The $N$-player game for $N\geq 3$ is analogous.
Let  $U_i, i=1,2 $, be compact metric spaces and $ V_i=\mathcal{P}(U_i)$, the space of probability measures on $U_i$ with Prohorov topology. Let
$$U:=U_1\times U_2 \; \; \mbox{and} \; \; V:=V_1\times V_2 .$$ 
Let $\bar{\pi}_{ij}: U \to [0,\infty) $ for $i \neq j$ and $\bar{\pi}_{ii}: U \to \mathbb{R}$ for $i \in S $. Define
$\pi_{ij} :  V \to \mathbb{R} $ as follows: for $v:=(v_1,v_2)\in V$,
\begin{equation*}
 \pi_{ij}(v_1,v_2)=\int_{U_2}\int_{U_1}\bar{\pi}_{ij}(u_1,u_2) v_1(du_1)v_2(du_2):=\int_U \bar{\pi}_{ij}(u) v(du), 
\end{equation*} 
where $\ u:=(u_1,u_2)\in U$, $i,j \in S:=\{1,2,\cdots \}$. Throughout this paper we assume that: \\
\noindent {\bf (A1)} 
 The transition rates $\bar{\pi}_{ij}(u) \geq 0$  for all $i \neq j, \ u \in  U$  and the transition rates  $\bar{\pi}_{ij}(u)$ are conservative, i.e.,  $$ \sum_{j \in S}\bar{\pi}_{ij}(u)=0~ \mbox{for} ~i \in S ~\mbox{and}~ u \in U \, .$$ 
 The functions $\bar{\pi}_{ij}$ are continuous and  
\[
\sup_{i \in S, u \in U} [-\bar{\pi}_{ii}(u)]:=M<\infty  \, .
\]
We consider a continuous time controlled  Markov chain  $Y(t)$ with
 state space $S$ and controlled rate matrix $\Pi_{v_1,v_2}=(\pi_{ij}(v_1,v_2))$, given by the stochastic integral
\begin{equation}\label{cmc2}
d Y(t) \ = \ \int_{\mathbb{R}} h(Y(t-),v_1(t),v_2(t), z) \wp(dz dt) .
\end{equation}
Here $\wp(dz dt)$ is a Poisson random measure with intensity $dz dt$, where $dz dt$ denotes the 
Lebesgue measure on $\mathbb{R}\times[0,\infty)$. The control process $v(\cdot):=( v_1(\cdot),v_2(\cdot))$ takes value in 
$V$, and  $h : S \times  V \times \mathbb{R} \to \mathbb{R}$ is defined as follows:
\begin{equation}\label{coefficient}
h(i, v, z) \ = \ 
\left\{
\begin{array}{lll}
j-i & {\rm if}& z \in \Delta_{ij}(v)\\
0& &{\rm otherwise} , \\ 
\end{array}
\right. 
\end{equation}
where $v:=(v_1, v_2)$ and  $\{ \Delta_{ij}(v) : i \neq j, \, i, j \in S\}$ denotes intervals of the form $[a, \ b)$
with length of $\Delta_{ij}(v) \ = \ \pi_{ij}(v)$ which are pairwise disjoint for each fixed $v \in  V$.\\
If $v_i(t)=\bar{v}_i(t,Y(t-))$ for some measurable 
 map $\bar{v}_i : [0, \infty) \times S \to  V_i$, then $v_i$ is called a Markov strategy for the ith player. With an abuse of terminology, the map $\bar{v}_i$ itself is called a Markov strategy of player $i$. A Markov strategy $v_i$ is called a stationary strategy if the map $\bar{v}_i$ does not have explicit dependence on time. We denote the  set of all Markov strategies by ${\mathcal M}_i$ and the set of all stationary strategies by $\mathcal{S}_i$ for the ith player. The spaces $\mathcal{S}_1$ and $\mathcal{S}_2$ are endowed with the product topology. Since $V_1$ and $V_2$ are compact, it follows that $\mathcal{S}_1$ and $\mathcal{S}_2$ are compact as well. 

The existence of a unique weak solution to the 
equation (\ref{cmc2}) for a pair of  Markov strategies $(v_1,v_2)$ for a given initial distribution
$\mu \in {\mathcal P}(S)$ follows using the assumption (A1), see Guo and Hern$\acute{\rm a}$ndez-Lerma [\cite{GuoLermabook}, Theorem 2.3, Theorem 2.5, pp.14-15].

We now list the commonly used notations below.
\begin{itemize}
\item $ C_b([a,b]\times S)$ denotes the set of all functions $ f:[a,b]\times S\to \mathbb R$ such that
 $ f(\cdot,i)\in C_b[a,b],\;\mbox{for each}\;\; i\in S $.
 \item $ C^1((a,b)\times S)$ denotes the set of all functions $ f:(a,b)\times S\to \mathbb R$ such that
 $ f(\cdot,i)\in C^1(a,b),\;\mbox{for each}\;\; i\in S $.
  \item $ C^{\infty}_c(a,b)$ denotes the set of all infinitely differentiable functions on $(a,b)$ with compact support.
 \item For any $f:S\to \mathbb{R}, \, (v_1,v_2)\in V_1 \times V_2, \, \Pi_{v_1,v_2} f(i)= \displaystyle{\sum_{j\in S}} \pi_{ij}(v_1,v_2)f(j).$
\end{itemize}
Set 
\[
 B_W(S) \ = \ \{ h : S \to \mathbb{R} | \sup_{i \in S} \frac{|h(i)|}{W(i)} < \infty \}, 
\]
where $W$ is the Lyapunov function as in (A3) (to be described in Section \ref{LF}).
 Define for $h \in B_W(S)$, 
\begin{equation*}\label{vn}
\|h\|_W \ = \ \sup_{i \in S} \frac{|h(i)|}{W(i)} \, .
\end{equation*}
 Then
$B_W(S)$ is a Banach space with the norm $\|\cdot\|_W$. 

For $k=1,2$, let $\bar{r}_k: S \times U_1 \times U_2 \rightarrow [0, \ \infty)$ be the running cost function for the $k$th player, i.e., when state of the system is $i$ and the actions $(u_1,u_2)$ are chosen by the players, then the $k$th player   incurs a cost at the rate of $\bar{r}_k(i,u_1,u_2)$. Throughout this paper, we assume that
the functions $\bar{r}_k$ are bounded and continuous.
Each player wants to minimize his accumulated cost over his strategies. The time horizon is infinite and we consider two risk sensitive cost evaluation criteria, viz., discounted cost and ergodic cost criteria which we describe now.
\subsection{Discounted cost criterion} Let $\theta_k \in (0, \ \Theta)$, for a fixed $\Theta > 0$, be the risk aversion parameter chosen by the $k$th player, $k=1,2$.  For a pair of Markov strategies $(v_1,v_2)$, the  $\alpha$-discounted payoff criterion, for kth
player is given by 
\begin{equation}\label{discountedcost}
\mathcal{J}^{v_1, v_2}_{\alpha, k}(\theta_k, i) \ :=\dfrac{1}{\theta_k} \ln \   E^{v_1, v_2}_i
\Big[ e^{\theta_k \int^{\infty}_0 e^{-\alpha t} r_k(Y(t-), v_1(t,Y(t-)), v_2(t,Y(t-)) dt} \Big] , i \in S, 
\end{equation}
where $\alpha > 0$ is the discount parameter, $Y(t)$ is the Markov chain corresponding to $(v_1,v_2) \in  \mathcal{M}_1 \times \mathcal{M}_2 $ and  $E_{i}^{v_1,v_2} $ denotes the expectation with respect to the law of the process $Y(t)$ with initial condition $Y(0)=i$, and  $r_k : S \times  V_1 \times  V_2 \to [0, \ \infty)$
is given by
\begin{equation*}
 r_k(i,v_1,v_2)=\int_{U_2}\int_{U_1}\bar{r}_k(i,u_1,u_2) v_1(du_1)v_2(du_2).
\end{equation*}
\begin{definitions}
%For fixed choices $(\theta_1, \theta_2)$ of the risk-sensitive parameters for the players 1 and 2,
For $(\theta_1,\theta_2) \in (0, \ \Theta)\times (0, \ \Theta)$, a pair of strategies $(v^*_1, v^*_2) \in {\mathcal M}_1 \times {\mathcal M}_2$ is said to be a
Nash equilibrium  if 
\begin{eqnarray*}
\mathcal{J}^{v^*_1, v^*_2}_{\alpha, 1}(\theta_1, i) & \leq & \mathcal{J}^{v_1, v^*_2}_{\alpha,1}(\theta_1, i) ,\; \mbox{for all} \; v_1\in \mathcal{M}_1\; \mbox{ and} \; i \in S \\
\mathcal{J}^{v^*_1, v^*_2}_{\alpha, 2}(\theta_2, i) & \leq & \mathcal{J}^{v^*_1, v_2}_{\alpha, 2}(\theta_2, i) , \ \mbox{for all} \; v_2\in \mathcal{M}_2\; \mbox{ and} \; i \in S.
\end{eqnarray*}
\end{definitions}
\subsection{Ergodic Cost Criterion}
For a pair of Markov strategies  $(v_1,v_2)$, the risk-sensitive ergodic cost for player $k$ is given by 
\begin{equation}\label{main1}
\rho^{v_1,v_2}_k (\theta_k,i)\ := \   \limsup_{ T \to \infty} \frac{1}{\theta_k T} \ln E_{i}^{v_1,v_2} \Big[ e^{\theta_k \int^T_0  r_k(Y(t-),v_1 (t,Y(t-)),v_2 (t,Y(t-))) dt} 
 \Big] \, .
\end{equation} 
%for some $\theta_k \in (0, \ \Theta)$, for a fixed $\Theta > 0$.  

\begin{definitions}
%For fixed choices $(\theta_1, \theta_2)$ of the risk-sensitive parameters for the players 1 and 2, 
For $(\theta_1,\theta_2) \in (0, \ \Theta)\times (0, \ \Theta)$, a pair of strategies $(v_1^{*},v_2^{*}) \in \mathcal{M}_1 \times \mathcal{M}_2$ is called a Nash equilibrium if 
\begin{equation*}
\rho_{1}^{v_1^*,v^*_2} (\theta_1,i)\ \leq \  \rho_{1}^{v_1,v^*_2} (\theta_1,i)\;\mbox{for all} \; v_1\in \mathcal{M}_1\; \mbox{ and} \; i \in S
\end{equation*}
and 
\begin{equation*}
\rho_{2}^{v_1^*,v^*_2} (\theta_2,i)\ \leq \  \rho_{2}^{v^*_1,v_2} (\theta_2,i)\;\mbox{for all} \; v_2\in \mathcal{M}_2\; \mbox{ and} \; i \in S.
\end{equation*}
\end{definitions}
\noindent We wish to establish the  existence of a Nash equilibrium in stationary strategies. 

We now outline a procedure for establishing the existence of a Nash equilibrium. We treat the ergodic cost case for this purpose. The discounted cost case can be analyzed along similar lines. Suppose player 2 announces that he is going to employ a strategy $v_2 \in \mathcal{S}_2$. In such a scenario, player 1 attempts to minimize
\begin{equation*}
\rho^{v_1,v_2}_1 (\theta_1,i)\ = \   \limsup_{ T \to \infty} \frac{1}{\theta_1 T} \ln E_{i}^{v_1,v_2} \Big[ e^{\theta_1 \int^T_0  r_1(Y(t-),v_1 (t,Y(t-)),v_2 (Y(t-))) dt} 
 \Big] \, ,
\end{equation*} 
over $v_1 \in \mathcal{M}_1$. Thus for player 1 it is a continuous time Markov decision problem (CTMDP) with risk sensitive ergodic cost. This problem has been studied by Ghosh and Saha \cite{GhoshSaha}, Kumar and Pal \cite{SP}, \cite{SP1}. In particular under certain assumptions, it is shown by Kumar and Pal \cite{SP}, \cite{SP1} that the following Hamilton-Jacobi-Bellman (HJB) equation
 \begin{equation*}
 \left\{\begin{aligned}
 \theta_1 \rho_{1} ~\hat{\psi}_{1}(i) &= \inf_{v_1\in V_1} \Big [ \Pi_{v_1,v_{2}(i)} 
\hat\psi_{1}(i) +\theta_1 r_1(i,v_1,v_{2}(i))\hat\psi_{1}(i) \Big ]   \\
 \hat\psi_{1}(i_0)  &= 1, 
 \end{aligned}
 \right.
\end{equation*}
has a suitable solution $(\rho_{1} ,\hat{\psi}_{1})$, where $\rho_{1} $ is a scalar and $\hat{\psi}_{1}:S \to \mathbb{R}$ has suitable growth rate; $i_0$ is a fixed element of $S$.  Furthermore it is shown by Kumar and Pal \cite{SP}, \cite{SP1} that
\begin{equation*}
\rho_1 \ = \  \inf_{v_1 \in \mathcal{M}_1} \limsup_{ T \to \infty} \frac{1}{\theta_1 T} \ln E_{i}^{v_1,v_2} \Big[ e^{\theta_1 \int^T_0  r_1(Y(t-),v_1 (t,Y(t-)),v_2 (Y(t-))) dt} 
 \Big] \, ,
\end{equation*}
and if $v_1^* \in \mathcal{S}_1$ is such that for $i \in S$
\begin{eqnarray*}
 && \inf_{v_1\in V_1} \Big [ \Pi_{v_1,v_{2}(i)} 
\hat\psi_{1}(i) +\theta_1 r_1(i,v_1,v_{2}(i))\hat\psi_{1}(i) \Big ]  \nonumber \\
&=&  \Pi_{v_{1}^*(i),v_{2}(i)} 
\hat\psi_{1}(i) +\theta_1 r_1(i,v_{1}^*(i),v_{2}(i))\hat\psi_{1}(i) ,
\end{eqnarray*}
then $v_1^* \in \mathcal{S}_1$ is an optimal control for player 1, i.e., for any $i \in S $ 
\begin{equation*}
\rho_1 \ = \   \limsup_{ T \to \infty} \frac{1}{\theta_1 T} \ln E_{i}^{v_1^*,v_2} \Big[ e^{\theta_1 \int^T_0  r_1(Y(t-),v_1^* (Y(t-)),v_2 (Y(t-))) dt} 
 \Big] \, .
\end{equation*}
 In other words given that player 2 is using the strategy $v_2 \in \mathcal{S}_2$, $v_1^* \in \mathcal{S}_1$ is an optimal response for player 1. Clearly $v_1^*$ depends on $v_2$ and moreover there may be several optimal responses for player 1 in $\mathcal{S}_1$. Analogous results holds for player 2 if player 1 announces that he is going to use a strategy $v_1 \in \mathcal{S}_1$. Hence given a pair of strategies $(v_1,v_2) \in \mathcal{S}_1 \times \mathcal{S}_2$, we can find a set of pairs of optimal responses $\{(v_1^*,v_2^*) \in \mathcal{S}_1 \times \mathcal{S}_2\}$ via the appropriate pair of HJB equations described above. This defines a set-valued map. Clearly any fixed point of this set-valued map is a Nash equilibrium.
 
 The above discussion leads to the following procedure for finding a pair of Nash equilibrium strategies for ergodic cost criterion.
 Suppose that there exist a pair of stationary strategies  $(v_1^*,v_2^*) \in \mathcal{S}_1 \times \mathcal{S}_2$, a pair of scalars $(\rho_1^*, \rho_2^*)$ and a pair of functions $(\hat{\psi}_{1}^*,\hat{\psi}_{2}^*)$ with appropriate growth conditions, such that the coupled HJB equations given by
 \begin{equation*}
 \left\{\begin{aligned}
        \theta_1 \rho^{*}_{1} ~\hat{\psi}^{*}_{1}(i) &= \inf_{v_1\in V_1} \Big [ \Pi_{v_1,v_{2}^*(i)} 
\hat\psi^{*}_{1}(i) +\theta_1 r_1(i,v_1,v_{2}^*(i))\hat\psi^{*}_{1}(i) \Big ]  \\
       &=  \Pi_{v_{1}^*(i),v_{2}^*(i)} 
\hat\psi^{*}_{1}(i) +\theta_1 r_1(i,v_{1}^*(i),v_{2}^*(i))\hat\psi^{*}_{1}(i)    \\
\displaystyle{ \hat\psi^{*}_{1}(i_0) } &= 1,  \\
\theta_2 \rho^{*}_{2} ~\hat{\psi}^{*}_{2}(i) &= \inf_{v_2\in V_2} \Big [ \Pi_{v_{1}^*(i),v_2} 
\hat\psi^{*}_{2}(i) +\theta_2 r_2(i,v_{1}^*(i),v_{2})\hat\psi^{*}_{2}(i) \Big ] \\
&=  \Pi_{v_{1}^*(i),v_{2}^*(i)} 
\hat\psi^{*}_{2}(i) +\theta_2 r_2(i,v_{1}^*(i),v_{2}^*(i))\hat\psi^{*}_{2}(i)   \\
\displaystyle{ \hat\psi^{*}_{2}(i_0) } &= 1,
       \end{aligned}
 \right.
\end{equation*}
where as before $i_0 \in S$ is a fixed point, then it can be shown that 
$(v_1^*,v_2^*)$ is a pair of Nash equilibrium and  $(\rho_1^*, \rho_2^*)$ is the pair of corresponding Nash values. An analogous coupled system of HJB equation for the discounted cost criterion can be derived along similar lines. We first solve the coupled HJB equations for the discounted cost criterion (to be describe in the next section). We then carry out the vanishing discount asymptotics to obtain an appropriate solution of the above coupled HJB equation for the ergodic cost criterion.
%Thus the main result of our paper is to establish that the above coupled HJB equations has suitable solutions. We achieve this by first studying the corresponding system of coupled HJB equation for the discounted case and then carry out the vanishing discount asymptotics.
\section{Analysis of Discounted Cost Criterion}
We carry out our analysis of the discounted cost criterion via the criterion 
\begin{equation}\label{discountedcost1}
J^{v_1, v_2}_{\alpha, k}(\theta_k, i) \ := \   E^{v_1, v_2}_i
\Big[ e^{\theta_k \int^{\infty}_0 e^{-\alpha t} r_k(X(t), v_1(t,Y(t-)), v_2(t,Y(t-)) dt} \Big] .
\end{equation}
Since logarithmic is an increasing function, therefore any Nash equilibrium for the criterion (\ref{discountedcost}) is Nash equilibrium for above criterion. To establish the existence of a Nash equilibrium for the discounted cost criterion, we first study the corresponding coupled  HJB equations.
 \subsection{Coupled HJB Equations for the Discounted Case}
 Let $v_2 \in \mathcal{S}_2$ be an arbitrarily fixed strategy of the second player. Consider the CTMDP for player 1 with the $\alpha$-discounted $(\alpha>0)$  risk-sensitive cost criterion  
\begin{equation}\label{discounted_cost1}
 J_{\alpha,1}^{v_2} (\theta_1,i,v_1)= E_i^{v_1,v_2}\left[ e^{\theta_1\int_{0}^{\infty} e^{-\alpha t} \;r_1(Y(t-), v_1(t,Y(t-)),v_2(Y(t-))) dt} \right],  
\end{equation}
 where $Y(t)$ is the process (\ref{cmc2}) corresponding to $(v_1,v_2) \in \mathcal{M}_1\times \mathcal{S}_2$ with  initial condition $i \in S$. \\
We define the value function for the cost criterion (\ref{discounted_cost1}) by
\begin{equation*}
 \psi_{\alpha,1}^{v_2}(\theta_1,i)= \inf_{v_1\in \mathcal M_1}  J_{\alpha,1}^{v_2} (\theta_1,i,v_1).
\end{equation*}
Then by the result of Ghosh and Saha \cite{GhoshSaha}, Kumar and Pal \cite{SP}, $\psi^{v_2}_{\alpha,1}$ is the unique  solution in $C_b([0,\Theta]\times S)\cap C^1((0,\Theta)\times S)$ to
 \begin{equation}\label{HJB1}
 \left\{\begin{aligned}
\alpha \theta_1  \dfrac{d\psi^{v_2}_{\alpha,1}}{d\theta_1}(\theta_1,i)&= \inf_{v_1\in V_1}\Big [ \Pi_{v_1,v_2(i)} 
\psi^{v_2}_{\alpha,1}(\theta_1,i) + \theta_1 r_1(i,v_1,v_2(i))\psi^{v_2}_{\alpha,1}(\theta_1,i) \Big ]   \\
\displaystyle{ \psi^{v_2}_{\alpha,1 }(0,i) } &= 1 .
 \end{aligned}
 \right.
\end{equation}

Similarly let player 1 fix a strategy $v_1 \in \mathcal{S}_1$ and consider the CTMDP for player 2 with $\alpha$-discounted risk-sensitive cost criterion  
\begin{equation*}
 J_{\alpha,2}^{v_1} (\theta_2,i,v_2)= E_i^{v_1,v_2}\left[ e^{\theta_2\int_{0}^{\infty} e^{-\alpha t} \;r_2(Y(t-), v_1(Y(t-)),v_2(t,Y(t-))) dt} \right].  
\end{equation*}
 Set 
\begin{equation*}
 \psi_{\alpha,2}^{v_1}(\theta_2,i)= \inf_{v_2\in \mathcal M_2}  J_{\alpha,2}^{v_1} (\theta_2,i,v_2).
\end{equation*}
Then, as before, $\psi^{v_1}_{\alpha,2}$ is the unique  solution in $C_b([0,\Theta]\times S)\cap C^1((0,\Theta)\times S)$ to
 \begin{equation}\label{HJB2}
 \left\{\begin{aligned}
\alpha \theta_2  \dfrac{d\psi^{v_1}_{\alpha,2}}{d\theta_2}(\theta_2,i)&= \inf_{v_2\in V_2}\Big [ \Pi_{v_1(i),v_2} 
\psi^{v_1}_{\alpha,2}(\theta_2,i) + \theta_2 r_2(i,v_1(i),v_2)\psi^{v_1}_{\alpha,2}(\theta_2,i) \Big ]   \\
\displaystyle{ \psi^{v_1}_{\alpha,2}(0,i) } &= 1 .
 \end{aligned}
 \right.
\end{equation}
To proceed further we establish some technical results needed later.  
\begin{lemma}\label{alphabound} Assume (A1). Then for each $\theta_k \in (0,\Theta)$, $k=1,2$, and $\alpha >0$, $i \in S $ $$1 \leq \psi^{v_1}_{\alpha,2}(\theta_2,i) \leq e^{\frac{\theta_2 \|r_2\|_{\infty}}{\alpha} },\; \; \mbox{and} \; \; 1 \leq \psi^{v_2}_{\alpha,1}(\theta_1,i) \leq e^{\frac{\theta_1 \|r_1\|_{\infty}}{\alpha} }. $$
Also 
$$ \| \dfrac{d\psi^{v_1}_{\alpha,2}}{d\theta_2}\|_{\infty} \leq \frac{ \|r_2\|_{\infty}}{\alpha} e^{\frac{\Theta \|r_2\|_{\infty}}{\alpha} }, \; \; \mbox{and} \; \; \|\dfrac{d \psi^{v_2}_{\alpha,1}}{d\theta_1}\|_{\infty} \leq \frac{ \|r_1\|_{\infty}}{\alpha}e^{\frac{\Theta \|r_1\|_{\infty}}{\alpha} }, $$
where $\|\cdot\|_{\infty}$ denotes the sup-norm.
\end{lemma}
\begin{proof}
Since
\begin{equation*}
 \psi_{\alpha,1}^{v_2}(\theta_1,i)= \inf_{v_1\in \mathcal M_1}   E_i^{v_1,v_2}\left[ e^{\theta_1\int_{0}^{\infty} e^{-\alpha t} \;r_1(Y(t-), v_1(t,Y(t-)),v_2(Y(t-))) dt} \right],
\end{equation*}
 it follows that
$$
1 \leq \psi_{\alpha,1}^{v_2}(\theta_1,i) \leq e^{\frac{\theta_1 \|r_1\|_{\infty} }{\alpha} },
\ \  \forall\; 0<\theta_1 < \Theta, i \in S ,\
0<\alpha<1.
$$
Similarly for $\psi_{\alpha,2}^{v_1}$ we obtain 
$$
1 \leq \psi_{\alpha,2}^{v_1}(\theta_2,i) \leq e^{\frac{\theta_2 \|r_2\|_{\infty} }{\alpha} },
\ \  \forall\; 0<\theta_2 < \Theta, i \in S ,\
0<\alpha<1.
$$
For $(v_1,v_2)\in \mathcal{M}_1 \times \mathcal{S}_2, \, i\in S $, set 
 $$
F_{\alpha}^1(i,v_1,v_2)=\int_{0}^{\infty} e^{-\alpha t} \;r_1(Y(t-), v_1(t,Y(t-)),v_2(Y(t-))) dt
$$ 
and 
$$
G_{\alpha}^1(\theta_1,i,v_1,v_2)=  E_i^{v_1,v_2}\left[ e^{\theta_1 \int_{0}^{\infty} e^{-\alpha t} \;r_1(Y(t-), v_1(t,Y(t-)),v_2(Y(t-))) dt}\right].
$$
 It is easily seen that 
$$
 \frac{d G_{\alpha}^1 }{ d \theta_1} = E_i^{v_1,v_2}[F_{\alpha}^1 e^{\theta_1 F_{\alpha}^1}] \leq
\frac{ \|r_1\|_{\infty}}{\alpha}e^{\frac{\Theta \|r_1\|_{\infty}}{\alpha} }.$$
 For each $\epsilon >0 $, 
$$
G_{\alpha}^1(\theta_1+\epsilon,i,v_1,v_2)-G_{\alpha}^1(\theta_1,i,v_1,v_2)= \epsilon \frac{d G_{\alpha}^1}{d \theta_1}
    (\theta_{\epsilon},i,v_1,v_2),
$$
for some $\theta_{\epsilon}$ which lies on the line segment joining $\theta_1$ and $\theta_1+\epsilon$.
Therefore 
$$
|G_{\alpha}^1(\theta_1+\epsilon,i,v_1,v_2)-G_{\alpha}^1(\theta_1,i,v_1,v_2)| \leq  \epsilon \frac{ \|r_1\|_{\infty}}{\alpha}e^{\frac{\Theta \|r_1\|_{\infty}}{\alpha} }.
$$
Thus 
\begin{eqnarray*}
| \psi_{\alpha,1}^{v_2}(\theta_1+\epsilon,i)-\psi_{\alpha,1}^{v_2}(\theta_1,i)| 
 &\leq & \sup_{v_1 \in \mathcal{M}_1} 
|G_{\alpha}^1(\theta_1+\epsilon,i,v_1,v_2)-G_{\alpha}^1(\theta_1,i,v_1,v_2)| \\
&\leq &  \epsilon \frac{ \|r_1\|_{\infty}}{\alpha}e^{\frac{\Theta \|r_1\|_{\infty}}{\alpha} }.
\end{eqnarray*}
Analogous bound can be proved if $\epsilon<0$. Hence it follows that 
$$
\Big \| \dfrac{d\psi^{v_2}_{\alpha,1}}{d\theta_1} (\theta_1,i) \Big \|_{\infty}
\leq \frac{ \|r_1\|_{\infty}}{\alpha} e^{\frac{\Theta \|r_1\|_{\infty}}{\alpha} }.
$$
Using analogous arguments we can show that 
$$
\Big \| \dfrac{d\psi^{v_1}_{\alpha,2}}{d\theta_2} (\theta_2,i) \Big \|_{\infty}
\leq \frac{ \|r_2\|_{\infty}}{\alpha} e^{\frac{\Theta \|r_2\|_{\infty}}{\alpha} }.
$$
This completes the proof.
\end{proof}
\begin{lemma}\label{continuity}
Assume (A1). Let $v_1 \in \mathcal{S}_1, v_2 \in \mathcal{S}_2. $ Then the maps $v_2 \to \psi^{v_2}_{\alpha,1}$ and $v_1 \to \psi^{v_1}_{\alpha,2}$ are continuous.
\end{lemma}
\begin{proof}
Let $v_2^m \to \hat{v}_2$ in $\mathcal{S}_2$, i.e., $v_2^m(i) \to \hat{v}_2(i)$ in $V_2$ for each $i \in S.$ By Lemma \ref{alphabound}, we have  $$ 1 \leq \psi^{v_2^m}_{\alpha,1}(\theta_1,i) \leq e^{\frac{\theta_1 \|r_1\|_{\infty}}{\alpha} } \; \; \mbox{and} \; \; \dfrac{ \psi^{v_2^m}_{\alpha,1}}{d\theta_1}(\theta_1,i) \leq \frac{ \|r_1\|_{\infty}}{\alpha}e^{\frac{\Theta \|r_1\|_{\infty}}{\alpha} }. $$
Thus by Arzela-Ascoli theorem there exists a function $\psi_{\alpha,1}$ in $C_b((0,\Theta]\times S)$ and a subsequence denoted by $\psi^{v_2^m}_{\alpha,1}$ by an abuse of notation, such that $\{\psi^{v_2^m}_{\alpha,1}\}$ converges uniformly to $\psi_{\alpha,1}$ over compact subset of $(0,\Theta)\times S $. 
 Let $\varphi \in C^{\infty}_c(0,\Theta)$. Then we have
\interdisplaylinepenalty=0
\begin{eqnarray*}
&&-\int_0^{\Theta} \alpha  \frac{ d (\theta_1 \varphi)  }{d \theta_1}  \psi^{v_2^m}_{\alpha,1} d\theta_1 =
\int_0^{\Theta} \alpha \theta_1 \frac{ d \psi^{v_2^m}_{\alpha,1} }{d \theta_1} \varphi d\theta_1  \\
&&= \int_0^{\Theta} \inf_{v_1\in V_1}\Big [ \Pi_{v_1,v_2^m(i)} 
\psi^{v_2^m}_{\alpha,1}(\theta_1,i) + \theta_1 r_1(i,v_1,v_2^m(i))\psi^{v_2^m}_{\alpha,1}(\theta_1,i) \Big ] \varphi d\theta_1 \\
&=&\int_0^{\Theta} \inf_{v_1\in V_1}\Big \{ \int_{U_1}\int_{U_2}\bar{\Pi}_{u_1,u_2} {v}_1(du_1)v_2^m(i)(du_2)
\psi^{v_2^m}_{\alpha,1}(\theta_1,i) \nonumber \\ 
&+&\theta_1 \int_{U_1}\int_{U_2} \bar{r}_1(i,u_1,u_2)){v}_1(du_1)v_2^m(i)(du_2) \psi^{v_2^m}_{\alpha,1}(\theta_1,i) \Big \}
\varphi d\theta_1 .
\end{eqnarray*}
 Letting $m \to \infty$  along a suitable subsequence and using (A1), we get for each $i \in S$, 
 \interdisplaylinepenalty=0
\begin{eqnarray*}
&&-\int_0^{\Theta} \alpha  \frac{ d (\theta_1 \varphi)  }{d \theta_1} \psi_{\alpha,1}(\theta_1, i) d\theta_1 \\
&=& \int_0^{\Theta} \inf_{v_1\in V_1}\Big \{ \int_{U_1}\int_{U_2}\bar{\Pi}_{u_1,u_2} {v}_1(du_1)\hat{v}_2(i)(du_2)
\psi_{\alpha,1}(\theta_1,i) \nonumber \\ 
&+&\theta_1 \int_{U_1}\int_{U_2} \bar{r}_1(i,u_1,u_2)){v}_1(du_1)\hat{v}_2(i)(du_2) \psi_{\alpha,1}(\theta_1,i) \Big \}
\varphi d\theta_1 \\
&=& \int_0^{\Theta} \inf_{v_1\in V_1}\Big [ \Pi_{v_1,\hat{v}_2(i)} 
\psi_{\alpha,1}(\theta_1,i) + \theta_1 r_1(i,v_1,\hat{v}_2(i))\psi_{\alpha,1}(\theta_1,i) \Big ] \Big \} \varphi (\theta) d\theta_1  .
\end{eqnarray*}
Therefore we have
\begin{eqnarray*}
  \alpha\theta_1 \dfrac{d \psi_{\alpha,1}}{d \theta_1} = 
 \inf_{v_1\in V_1}\Big [ \Pi_{v_1,\hat{v}_2(i)} 
\psi_{\alpha,1}(\theta_1,i) + \theta_1 r_1(i,v_1,\hat{v}_2(i))\psi_{\alpha,1}(\theta_1,i) \Big ] ,
\end{eqnarray*}
in the sense of distribution.  Note that right hand side above is continuous. Therefore
 $\frac{d \psi_{\alpha,1}}{d \theta} \in C((0,\Theta)\times S)$.
Thus $\psi_{\alpha,1} \in C_b([0,\Theta]\times S)\cap C^1((0,\Theta)\times S)$ is a solution to
\interdisplaylinepenalty=0
 \begin{equation*}
 \left\{\begin{aligned}
\alpha \theta_1  \dfrac{d\psi_{\alpha,1}}{d\theta_1}(\theta_1,i)&= \inf_{v_1\in V_1}\Big [ \Pi_{v_1,\hat{v}_2 (i)} 
\psi_{\alpha,1}(\theta_1,i) + \theta_1 r_1(i,v_1,\hat{v}_2(i))\psi_{\alpha,1}(\theta_1,i) \Big ]  \nonumber \\
\displaystyle{ \psi_{\alpha,1 }(0,i) } &= 1 .
 \end{aligned}
 \right.
\end{equation*}
 Therefore, using It$\hat{{\rm o}}$'s formula,  $\psi_{\alpha,1}$ admits the following representation 
 \begin{equation*}
 \psi_{\alpha,1}(\theta_1,i)= \inf_{v_1\in \mathcal M_1}  E_i^{v_1,\hat{v}_2}\left[ e^{\theta_1\int_{0}^{\infty} e^{-\alpha t} \;r_1(Y(t-), v_1(t,Y(t-)),\hat{v}_2(Y(t-))) dt} \right].
\end{equation*}
Hence  $\psi_{\alpha,1}= \psi^{\hat{v}_2}_{\alpha,1}$. The continuity of $v_2 \to \psi^{v_2}_{\alpha,1}$ follows. Similarly we can show that  $v_1 \to \psi^{v_1}_{\alpha,1}$ is continuous. This completes the proof. 
\end{proof}
\begin{theorem}\label{fixed}
Assume (A1). Then there exist $(v_1^*,v_2^*)\in \mathcal{S}_1 \times \mathcal{S}_2 $ and a pair of functions which are bounded and continuously differentiable $ (\psi^{v_2^*}_{\alpha,1},\psi^{v_1^*}_{\alpha,2} )$ satisfying the following coupled HJB equations
\interdisplaylinepenalty=0
\begin{equation}\label{coupled-hjb}
 \left\{\begin{aligned}
        \alpha \theta_1  \dfrac{d\psi^{v_2^*}_{\alpha,1}}{d\theta_1}(\theta_1,i)&= \inf_{v_1\in V_1}\Big [ \Pi_{v_1,v_2^*(i)} 
\psi^{v_2^*}_{\alpha,1}(\theta_1,i) + \theta_1 r_1(i,v_1,v_2^*(i))\psi^{v_2^*}_{\alpha,1}(\theta_1,i) \Big ] \, \,   \\
&=  \Pi_{v_1^*,v_2^*(i)} 
\psi^{v_2^*}_{\alpha,1}(\theta_1,i) + \theta_1 r_1(i,v_1^*,v_2^*(i))\psi^{v_2^*}_{\alpha,1}(\theta_1,i)   \\
\displaystyle{ \psi^{v_2^*}_{\alpha,1 }(0,i) } &= 1 ,  \\
\alpha \theta_2  \dfrac{d\psi^{v_1^*}_{\alpha,2}}{d\theta_2}(\theta_2,i)&= \inf_{v_2\in V_2}\Big [ \Pi_{v_1^*(i),v_2} 
\psi^{v_1^*}_{\alpha,2}(\theta_2,i) + \theta_2 r_2(i,v_1^*(i),v_2)\psi^{v_1^*}_{\alpha,2}(\theta_2,i) \Big ] \, \,   \\
&=  \Pi_{v_1^*(i),v_2^*} 
\psi^{v_1^*}_{\alpha,2}(\theta_2,i) + \theta_2 r_2(i,v_1^*(i),v_2^*)\psi^{v_1^*}_{\alpha,2}(\theta_2,i)   \\
\displaystyle{ \psi^{v_1^*}_{\alpha,2}(0,i) } &= 1 .
       \end{aligned}
 \right.
\end{equation}
\end{theorem}
%First note that $v_1^*\in \mathcal{S}_1  $ since $\theta_1$ is fixed.
%We want to show that there exists a pair of strategies $(v_1^*,v_2^*)$ which satisfies the above equation. 
\begin{proof}
Let $v_2 \in \mathcal{S}_2$. For $i\in S, \, v_1 \in V_1$, set
$$F_1(i,v_1,v_2(i))=\Pi_{v_1,v_2(i)} 
\psi^{v_2}_{\alpha,1}(\theta_1,i) +\theta_1 r_1(i,v_1,v_2(i))\psi^{v_2}_{\alpha,1}(\theta_1,i),$$ where $\psi^{v_2}_{\alpha,1}$ is the solution of the equation (\ref{HJB1}).
Let $$H(v_2)=\Big \{ v_1^* \in \mathcal{S}_1 \Big |  F_1(i,v_1^*(i),v_2(i))=\inf_{v_1 \in V_1} F_1(i,v_1,v_2(i)) \, \mbox{for all} \, \, i \in S \Big \}.$$
Obviously $v_1^*$ depends on $\theta_1$ and $\alpha$. We suppress  this dependence for notational simplicity.
Then by a standard  measurable selection theorem (see Bene$\check{s}$ \cite{Benes}), $H(v_2)$ is a non empty subset of $\mathcal{S}_1$. Clearly $H(v_2)$ is convex. It is easy to show that $H(v_2)$ is closed and hence compact. 

Similarly for $i\in S$, $v_1 \in \mathcal{S}_1$ and $v_2 \in V_2$, set
$$F_2(i,v_1(i),v_2)=\Pi_{v_1(i),v_2} 
\psi^{v_1}_{\alpha,2}(\theta_2,i) +\theta_2 r_2(i,v_1(i),v_2)\psi^{v_1}_{\alpha,2}(\theta_2,i),$$  where $\psi^{v_1}_{\alpha,2}$ is the solution of the equation (\ref{HJB2}).
Let $$H(v_1)=\Big \{ v_2^* \in \mathcal{S}_2 \Big |  F_2(i,v_1(i),v_2^*(i))=\inf_{v_2 \in V_2} F_2(i,v_1(i),v_2)  \, \mbox{for all} \, \, i \in S  \Big \}.$$
Then, as before, $H(v_1)$ is convex and compact.\\
Define $$H(v_1,v_2)=H(v_2) \times H(v_1).$$
Then $ H(v_1,v_2)$ is nonempty, convex, and compact subset of $\mathcal{S}_1\times \mathcal{S}_2$ . Thus $$(v_1,v_2) \to H(v_1,v_2)$$ defines a point to set map from $\mathcal{S}_1\times \mathcal{S}_2$ to $2^{\mathcal{S}_1\times \mathcal{S}_2}$. Next we want to show that this map is upper semicontinuous. 
Let $\{(v_1^m,v_2^m)\}\in \mathcal{S}_1\times \mathcal{S}_2$ and $(v_1^m,v_2^m) \to (\hat{v_1},\hat{v}_2)$ in $\mathcal{S}_1\times \mathcal{S}_2$, i.e., for each $i \in S, \, (v_1^m(i),v_2^m(i)) \to (\hat{v_1}(i),\hat{v}_2(i)) \, \, \mbox{in} \, \, V_1\times V_2  $. Let $\bar{v}_1^m \in H(v_2^m)$. Then $\{\bar{v}_1^m \}\subset \mathcal{S}_1$. Since $\mathcal{S}_1$ is compact, it has a convergent subsequence, denoted by the same sequence with an abuse of notation, such that $$\bar{v}_1^m \to \bar{v}_1 \; \; \mbox{in} \; \;\mathcal{S}_1.$$ Then $(\bar{v}_1^m, v_2^m) \to (\bar{v}_1, \hat{v}_2) \, \, \mbox{in} \, \, \mathcal{S}_1\times \mathcal{S}_2 $ . Now using Lemmata \ref{alphabound} and \ref{continuity}, we obtain 
\interdisplaylinepenalty=0
\begin{eqnarray}
F_1(i,\bar{v}_1(i),\hat{v}_2(i))
&=&\int_{U_1}\int_{U_2}\bar{\Pi}_{u_1,u_2} \bar{v}_1(i)(du_1)\hat{v}_2(i)(du_2)
\psi^{\hat{v}_2}_{\alpha,1}(\theta_1,i) \nonumber \\ 
&+&\theta_1 \int_{U_1}\int_{U_2} \bar{r}_1(i,u_1,u_2)\bar{v}_1(i)(du_1)\hat{v}_2(i)(du_2) \psi^{\hat{v}_2}_{\alpha,1}(\theta_1,i) \nonumber \\
&=&\sum_{j \in S}\int_{U_1}\int_{U_2}\bar{\pi}_{ij}(u_1,u_2)\psi^{\hat{v}_2}_{\alpha,1}(\theta_1,j) \bar{v}_1(i)(du_1)\hat{v}_2(i)(du_2)
 \nonumber \\ 
&+&\theta_1 \int_{U_1}\int_{U_2} \bar{r}_1(i,u_1,u_2)\bar{v}_1(i)(du_1)\hat{v}_2(i)(du_2) \psi^{\hat{v}_2}_{\alpha,1}(\theta_1,i) \nonumber \\
&=&\lim_{m\to \infty} \sum_{j \in S}\int_{U_1}\int_{U_2}\bar{\pi}_{ij}(u_1,u_2)\psi^{v_2^m}_{\alpha,1}(\theta_1,j) \bar{v}_1^m(i)(du_1) v_2^m(i)(du_2)
 \nonumber \\ 
&+&\theta_1 \lim_{m\to \infty} \int_{U_1}\int_{U_2} \bar{r}_1(i,u_1,u_2)\bar{v}_1^m(du_1)v_2^m(i)(du_2) \psi^{v_2^m}_{\alpha,1}(\theta_1,i) \nonumber \\
&=&\lim_{m\to \infty}F_1(i,\bar{v}_1^m(i),v_2^m(i)).
\end{eqnarray}
Now fix $\tilde{v}_1 \in \mathcal{S}_1$ and consider the sequence $(\tilde{v}_1,v_2^m)$. Repeat the above argument to conclude that
\begin{eqnarray}
F_1(i,\tilde{v}_1(i),\hat{v}_2(i))
=\lim_{m\to \infty}F_1(i,\tilde{v}_1(i),v_2^m(i)). 
\end{eqnarray}
Using the fact $\bar{v}_1^m \in H(v_2^m)$, we have 
$$ F_1(i,\tilde{v}_1(i),v_2^m(i)) \geq F_1(i,\bar{v}_1^m(i),v_2^m(i)) \; \mbox{for all}\; m. $$
Thus $$F_1(i,\tilde{v}_1(i),\hat{v}_2(i)) \geq F_1(i,\bar{v}_1(i),\hat{v}_2(i))\; \mbox{for all}\; \tilde{v}_1 \in \mathcal{S}_1. $$
Therefore $\bar{v}_1 \in H(\hat{v}_2)$. Let $\bar{v}_2^m \in H(v_1^m)$ and along a subsequence  $$\bar{v}_2^m \to \bar{v}_2 \; \; \mbox{in} \; \;\mathcal{S}_2.$$ Using analogous arguments, we obtain  $\bar{v}_2 \in H(\hat{v}_1)$. The upper semicontinuity of the map $$(v_1,v_2) \to H(v_1,v_2)$$ follows. By Fan's fixed point theorem \cite{Fan}, there exists $(v_1^*, v_2^*) \in \mathcal{S}_1 \times \mathcal{S}_2$ such that $(v_1^*, v_2^*) \in H(v_1^*, v_2^*)$.  This establishes that there exist $(v_1^*,v_2^*)\in \mathcal{S}_1 \times \mathcal{S}_2 $  and $ (\psi^{v_2^*}_{\alpha,1},\psi^{v_1^*}_{\alpha,2} )$ satisfy the coupled HJB equations (\ref{coupled-hjb}). 
\end{proof}
\begin{remark}
 (i) Note that $v_1^*$ and $v_2^*$ depend on $\theta_1$ and $\theta_2$ respectively. Hence $ \psi^{v_2^*}_{\alpha,1}$ (resp. $\psi^{v_1^*}_{\alpha,2} $) depend both on $\theta_1$ and $\theta_2$. Thus for each $\theta_2$, $ \psi^{v_2^*}_{\alpha,1}$  is continuously differentiable with respect to $\theta_1$  and for each $\theta_1$, $ \psi^{v_1^*}_{\alpha,2}$  is continuously differentiable with respect to $\theta_2$. We have suppressed this dependence on $\theta_2$ (resp. on $\theta_1$) for notational convenience.\\
 (ii) Note that for the discounted cost criterion the corresponding coupled HJB equations are given by (\ref{coupled-hjb}). However, the pair of stationary strategies $(v_1^*,v_2^*)$ does not constitute a Nash equilibrium for this criterion. If player 1 announces his strategy $v_1^*$ then the optimal response for player 2 for the discounted criterion is given by the the Markov strategy $v_2^*(\theta_2 e^{-\alpha t},i)$. An analogous statement holds for optimal response of player 1. Thus the existence of a pair of Nash equilibrium in Markov strategies for the discounted cost criterion needs further analysis which we carry out in the next subsection.
 \end{remark}
 \subsection{Existence of Nash Equilibrium for Discounted Cost Criterion} In this subsection we establish the existence of a Nash equilibrium under the following additive structure on $\bar{\pi}_{u_1,u_2}$ and $\bar{r}_k(i,u_1,u_2)$.
 
  \noindent {\bf (A2)} We assume that $\bar{\pi}_{ij} :  U_1 \times U_2 \to \mathbb{R}$
and $\bar{r}_k : S \times U_1 \times U_2 \to [0, \ \infty), \ k =1,2$ are given by
\begin{eqnarray*}
\bar{\pi}_{ij}( u_1, u_2) & = & \bar{\pi}_{ij}^1( u_1) + \bar{\pi}_{ij}^2(u_2), \\
\bar{r}_k (i, u_1, u_2) & = & \bar{r}_{k1}(i, u_1) + \bar{r}_{k2}(i, u_2), i \in S, u_1 \in U_1,
u_2 \in U_2, k =1,2,
\end{eqnarray*}
 where $\bar{\pi}_{ij}^k :  U_k  \to \mathbb{R}$ assumed to be continuous; 
\[
\sup_{i,j \in S, u \in U} |\bar{\pi}_{ij}^k(u)|:=M<\infty  \, .
\]
$\bar{r}_{kl} : S \times U_l  \to [0, \ \infty), \ k,l =1,2$, assumed to be bounded and continuous.

These type of additive structure are rather standard in non-zero sum stochastic differential games (see, e.g., Borkar and Ghosh \cite{Borkar-Ghosh}) and non-zero sum stochastic games on an uncountable state space by Himmelberg et al. \cite{Himmelberg}. In fact in stochastic games these conditions are referred to as ARAT (additive reward, additive transition).

Now we define a class of strategies to be referred to as eventually stationary strategies denoted by $\hat{\mathcal S}_k,\; k=1,2$. Let 
\[
\hat{\mathcal S}_k \ = \ \{ \hat{v}_k : (0, \Theta) \times S \to V_k \,  | \hat{v}_k(\cdot,i) \ {\rm is\ measurable\ for\ each}\; i\in S \}, \ 
k=1,2.
\]
Note that as opposed to $\mathcal{S}_k$, the topology of pointwise convergence on $\hat{\mathcal{S}}_k$ is not metrizable. Thus we endow the space $\hat{\mathcal{S}}_k$ with
 the weak* topology on $L^\infty ((0, \Theta) \times S, {\mathcal M}_s(U_k)), 
\ k=1,2$, introduced by Warga \cite{Warga} for the topology of relaxed controls, where $\mathcal{M}_s(U_k)$ is the space of all finite signed measure on $U_k$ endowed with the topology of weak convergence. Note that with the above topology, $\hat{\mathcal S}_k$ becomes a compact metrizable space with following convergence criterion:\\
For $k =1,2$, 
$\hat{v}^n_k \to \hat{v}_k$ in $\hat{\mathcal S}_k$ as $n \to \infty$  if and only if for each $i \in S$
\begin{equation}\label{convergencecriterion1}
\lim_{n \to \infty} \int_0^\Theta f(\theta) \int_{U_k} g(\theta, u_k) \hat{v}^n_k(\theta,i)(du_k) d \theta  
\ = \  \int_0^\Theta f(\theta) \int_{U_k} g(\theta, u_k) \hat{v}_k(\theta,i)(du_k) d \theta  ,
\end{equation}
for all $f \in L^1(0, \Theta) \cap L^2(0, \Theta)=L^2(0, \Theta)  , 
\ g \in C_b((0, \Theta) \times U_k)$. The Markov strategies associated with 
$\hat{v}_k \in \hat{\mathcal S}_k, k=1,2$ is given by $\hat{v}_k(\theta e^{-\alpha t}, Y(t-)), t \geq 0$ for each 
$\theta \in (0, \Theta)$ and $\alpha > 0$, where $Y(t)$ is the solution of the equation
\begin{equation}\label{cmc21}
d Y(t) \ = \ \int_{\mathbb{R}} h(Y(t-), \hat{v}_1(\theta e^{-\alpha t}, Y(t-)), \hat{v}_2(\theta e^{-\alpha t}, Y(t-)), z) \wp(dz dt) .
\end{equation}
By an abuse of notation, we represent the eventually stationary Markov strategies by elements of 
$\hat{\mathcal S}_k$ though each member in $\hat{\mathcal S}_k$ corresponds to a family
of Markov strategies indexed by $\theta$ and $\alpha$. Note that as $t \to \infty, \; e^{-\alpha t} \to 0$. Thus in the long run an element of $\hat{\mathcal S}_k$ ``eventually" becomes an element of $\mathcal S_k$. Hence the terminology.

  Define
for $\hat{v}_k \in \hat{\mathcal S}_k, \ k =1,2$ 
\begin{eqnarray}\label{valuefunctions}
\tilde{\psi}^{\hat{v}_2}_{\alpha, 1}(\theta_1, i) & = & \inf_{\tilde{v}_1 \in {\mathcal M}_1} J^{\tilde{v}_1, \hat{v}_2}_{\alpha, 1}
(\theta_1, i), \ \theta_1, \in (0, \ \Theta), \ i \in S \nonumber \\
\tilde{\psi}^{\hat{v}_1}_{\alpha, 2}(\theta_2, i) & = & \inf_{\tilde{v}_2 \in {\mathcal M}_2} J^{\hat{v}_1, \tilde{v}_2}_{\alpha, 2}
(\theta_2, i), \  \theta_2, \in (0, \ \Theta), \ i \in S.
\end{eqnarray}
%Now we state and prove the following theorem which give a characterization of the above value functions.
By using similar arguments as in previous subsection it follows that $\tilde{\psi}^{\hat{v}_2}_{\alpha, 1}$ is a bounded and absolutely continuous function satisfying the the following equation 
%$C_b([0,\Theta]\times S)\cap C^1((0,\Theta)\times S)$ to 
%By closely mimicking the arguments In [\cite{GhoshSaha}, Theorem 3.4] we obtain the following theorem. We omit the details.
%For each $v_2 \in \hat{\mathcal S}_2$, the o.d.e. 
\begin{equation}\label{discountedhjb11}
 \left\{\begin{aligned}
\alpha \theta_1  \dfrac{d\tilde{\psi}^{\hat{v}_2}_{\alpha, 1}}{d\theta_1}(\theta_1,i)&= \inf_{v_1\in V_1}\Big [ \Pi_{v_1,\hat{v}_2(\theta_1, i)} 
\tilde{\psi}^{\hat{v}_2}_{\alpha, 1}(\theta_1,i) + \theta_1 r_1(i,v_1,\hat{v}_2(\theta_1, i))\tilde{\psi}^{\hat{v}_2}_{\alpha, 1}(\theta_1,i) \Big ]   \\
\displaystyle{ \tilde{\psi}^{\hat{v}_2}_{\alpha, 1}(0,i) } &= 1 ,
 \end{aligned}
 \right.
\end{equation}
 and $\tilde{\psi}^{\hat{v}_1}_{\alpha, 2}$ is a bounded and absolutely continuous function satisfying the the following equation  
%$C_b([0,\Theta]\times S)\cap C^1((0,\Theta)\times S)$ to 
%(ii) Similarly, For each $v_1 \in \hat{\mathcal S}_1$, the o.d.e. 
 \begin{equation}\label{discountedhjb22}
 \left\{\begin{aligned}
\alpha \theta_2  \dfrac{d\tilde{\psi}^{\hat{v}_1}_{\alpha, 2}}{d\theta_2}(\theta_2,i)&= \inf_{v_2\in V_2}\Big [ \Pi_{\hat{v}_1(\theta_2,i),v_2} 
\tilde{\psi}^{\hat{v}_1}_{\alpha, 2}(\theta_2,i) + \theta_2 r_2(i,\hat{v}_1(\theta_2,i),v_2)\tilde{\psi}^{\hat{v}_1}_{\alpha, 2}(\theta_2,i) \Big ].   \\
\displaystyle{ \tilde{\psi}^{\hat{v}_1}_{\alpha, 2}(0,i) } &= 1 .
 \end{aligned}
 \right.
\end{equation}
As before we can establish the following result, we omit the details.
\begin{lemma} \label{alphabound1}
Assume (A1). Then for $\theta \in (0, \Theta)$ and $\alpha > 0$ and $\hat{v}_k \in \hat{\mathcal S}_k, k =1,2$, we have
\begin{eqnarray}\label{eq1estimates}
1 \leq \max \{ \tilde{\psi}^{\hat{v}_2}_{\alpha, 1} (\theta, i) , \tilde{\psi}^{\hat{v}_1}_{\alpha, 2} (\theta, i) \} 
\leq \ \max_{k=1,2} \{ e^{\frac{\theta \|r_k\|_{\infty}}{\alpha}}\}, & \\ \nonumber 
\max \{ \|\frac{d \tilde{\psi}^{\hat{v}_2}_{\alpha, 1}}{d \theta} \|_{\infty}, 
\|\frac{d \tilde{\psi}^{\hat{v}_1}_{\alpha, 2}}{d \theta} \|_{\infty} \} 
 \ \leq \ \max_{k=1,2} \{\frac{\|r_k\|_{\infty}}{\alpha} e^{\frac{\Theta \| r_k\|_{\infty}}{\alpha}}\} . & 
\end{eqnarray}

\end{lemma}

\begin{lemma}\label{continuity1} Assume (A1). Then the maps $ \hat{v}_1 \mapsto \tilde{\psi}^{\hat{v}_1}_{\alpha, 2} $ 
from $\hat{\mathcal S}_1 \to C^{1}([0, \Theta]  \times S)$  
and $ \hat{v}_2 \mapsto \tilde{\psi}^{\hat{v}_2}_{\alpha, 1} $
 from $\hat{\mathcal S}_2  \to C^{1}([0, \Theta]  \times S)$ are continuous.
\end{lemma}

Let $\hat{v}_2 \in \hat{\mathcal{S}}_2$. For $i\in S, \, v_1 \in V_1$ and  $\theta_1 \in (0, \Theta)$, set
$$\tilde{F}_1(i,v_1,\hat{v}_2(\theta_1,i))=\Pi_{v_1,\hat{v}_2(\theta_1,i)} 
\tilde{\psi}^{\hat{v}_2}_{\alpha,1}(\theta_1,i) +\theta_1 r_1(i,v_1,\hat{v}_2(\theta_1,i)))\tilde{\psi}^{\hat{v}_2}_{\alpha,1}(\theta_1,i).$$
Let $$\tilde{H}(\hat{v}_2)=\Big \{ \hat{v}_1^* \in \hat{\mathcal{S}}_1 \Big |  \tilde{F}_1(i,\hat{v}_1^*(\theta_1,i),\hat{v}_2(\theta_1,i))=\inf_{v_1 \in V_1} \tilde{F}_1(i,v_1,\hat{v}_2(\theta_1,i)) \  \mbox{for all} \, \, i \in S \Big \}.$$
%\mbox{ a.e.} \; \theta_1, \,
%Then by  \cite{Benes}, $H(v_2)$ is a non empty subset of $\hat{\mathcal{S}}_1$. It is easy to show that $H(v_2)$ is convex and compact. 

Similarly for $i\in S$, $\hat{v}_1 \in \hat{\mathcal{S}}_1$, $v_2 \in V_2$ $\theta_2 \in (0, \Theta)$, set
$$\tilde{F}_2(i,\hat{v}_1(\theta_2,i),v_2)=\Pi_{\hat{v}_1(\theta_2,i),v_2} 
\tilde{\psi}^{\hat{v}_1}_{\alpha,2}(\theta_2,i) +\theta_2 r_2(i,\hat{v}_1(\theta_2,i),v_2)\tilde{\psi}^{\hat{v}_1}_{\alpha,2}(\theta_2,i),$$
and $$\tilde{H}(\hat{v}_1)=\Big \{\hat{v}_2^* \in \hat{\mathcal{S}}_2 \Big |  \tilde{F}_2(i,\hat{v}_1(\theta_2,i),\hat{v}_2^*(\theta_2,i))=\inf_{v_2 \in V_2} \tilde{F}_2(i,\hat{v}_1(\theta_2,i),v_2) \   \, \mbox{for all} \, \, i \in S  \Big \}.$$
%Then, as before, $H(v_1)$ is convex and compact.\\
Define $$\tilde{H}(\hat{v}_1,\hat{v}_2)=\tilde{H}(\hat{v}_2) \times \tilde{H}(\hat{v}_1).$$
Then using arguments as in Theorem \ref{fixed}, it follows that $ \tilde{H}(\hat{v}_1,\hat{v}_2)$ is nonempty, convex, and compact subset of $\hat{\mathcal{S}}_1\times \hat{\mathcal{S}}_2$. Therefore $(\hat{v}_1,\hat{v}_2) \mapsto \tilde{H}(\hat{v}_1,\hat{v}_2) $ defines a map from $\hat{\mathcal S}_1 \times \hat{\mathcal S}_2 \to 
2^{\hat{\mathcal S}_1} \times 2^{\hat{\mathcal S}_2}$. Now we establish the following result.

\begin{lemma}\label{usc1} Assume (A1) and (A2). Then the map 
$(\hat{v}_1,\hat{v}_2) \mapsto \tilde{H}(\hat{v}_1,\hat{v}_2) $  is upper semicontinuous.
\end{lemma}
\begin{proof}
Let $\{(v_1^m,v_2^m)\}\in \hat{\mathcal{S}}_1\times \hat{\mathcal{S}}_2$ and $(v_1^m,v_2^m) \to (\hat{v_1},\hat{v}_2)$ in $\hat{\mathcal{S}}_1\times \hat{\mathcal{S}}_2$. Let $\bar{v}_1^m \in \tilde{H}(v_2^m)$. Then ${\bar{v}_1^m}\subset \hat{\mathcal{S}}_1$. Since $\hat{\mathcal{S}}_1$ is compact, it has a convergent subsequence, denoted by the same sequence with an abuse of notation, such that $$\bar{v}_1^m \to \bar{v}_1 \; \; \mbox{in} \; \;\hat{\mathcal{S}}_1.$$ Then $(\bar{v}_1^m, v_2^m) \to (\bar{v}_1, \hat{v}_2) \, \, \mbox{in} \, \, \hat{\mathcal{S}}_1\times \hat{\mathcal{S}}_2 $ . Now using (A2), Lemmata \ref{alphabound1} and \ref{continuity1}, and the topology of $\hat{\mathcal{S}}_k,k=1,2$ it follows that for each $i\in S$
$$  \Pi_{\bar{v}_1^m,v_2^m(\theta_1,i)} 
\tilde{\psi}^{v_2^m}_{\alpha,1}(\theta_1,i) +\theta_1 r_1(i,\bar{v}_1^m,v_2^m(\theta_1,i)))\tilde{\psi}^{v_2^m}_{\alpha,1}(\theta_1,i)$$
converges weakly in $L^2(0,\Theta)$ to 
$$  \Pi_{\bar{v}_1,\hat{v}_2(\theta_1,i)} 
\tilde{\psi}^{\hat{v}_2}_{\alpha,1}(\theta_1,i) +\theta_1 r_1(i,\bar{v}_1,\hat{v}_2(\theta_1,i)))\tilde{\psi}^{\hat{v}_2}_{\alpha,1}(\theta_1,i).$$
Thus, by Banach-Saks theorem any sequence of convex combination of the former converges strongly in $L^2(0,\Theta)$ to the latter. Hence along a suitable subsequence 
\interdisplaylinepenalty=0
\begin{eqnarray*}
\lim_{m\to \infty}\tilde{F}_1(i,\bar{v}_1^m(\theta_1,i),v_2^m(\theta_1,i))
=\tilde{F}_1(i,\bar{v}_1(\theta_1,i),\hat{v}_2(\theta_1,i)), \; \mbox{a.e.\ in} \ \theta_1.
\end{eqnarray*}
Now fix $\tilde{v}_1 \in \hat{\mathcal{S}}_1$ and consider the sequence $(\tilde{v}_1,v_2^m)$. Repeat the above argument to conclude that
\begin{eqnarray*}
\tilde{F}_1(i,\tilde{v}_1(\theta_1,i),\hat{v}_2(\theta_1,i))
=\lim_{m\to \infty}\tilde{F}_1(i,\tilde{v}_1(\theta_1,i),v_2^m(\theta_1,i)), \; \mbox{a.e.\ in} \ \theta_1. 
\end{eqnarray*}
Using the fact $\bar{v}_1^m \in H(v_2^m)$, for any $m$ we have 
$$ \tilde{F}_1(i,\tilde{v}_1(\theta_1,i),v_2^m(\theta_1,i)) \geq \tilde{F}_1(i,\bar{v}_1^m(\theta_1,i),v_2^m(\theta_1,i)), \; \mbox{a.e.\ in} \ \theta_1. $$
Thus for any $\tilde{v}_1 \in \hat{\mathcal{S}}_1$ $$\tilde{F}_1(i,\tilde{v}_1(\theta_1,i),\hat{v}_2(\theta_1,i)) \geq \tilde{F}_1(i,\bar{v}_1(\theta_1,i),\hat{v}_2(\theta_1,i)), \; \mbox{a.e.\ in} \ \theta_1. $$
Therefore $\bar{v}_1 \in \tilde{H}(\hat{v}_2)$. Let $\bar{v}_2^m \in \tilde{H}(v_1^m)$ and along a subsequence  $$\bar{v}_2^m \to \bar{v}_2 \; \; \mbox{in} \; \;\hat{\mathcal{S}}_2.$$ Using analogous arguments, we obtain  $\bar{v}_2 \in \tilde{H}(\hat{v}_1)$. This prove that the map $$(\hat{v}_1,\hat{v}_2) \mapsto \tilde{H}(\hat{v}_1,\hat{v}_2) $$ is upper semicontinuous. 
\end{proof}
 
\begin{theorem}\label{Nashequilibriumdiscounted} Assume (A1) and (A2). There exists $\alpha$-discounted
Nash equilibrium in the class $\hat{\mathcal S}_1 \times \hat{\mathcal S}_2$.
\end{theorem}

\begin{proof} 
From Lemma \ref{usc1} and Fan's fixed point theorem \cite{Fan}, there exists a fixed point 
$(\hat{v}^*_1, \hat{v}^*_2) \in \hat{\mathcal S}_1 \times \hat{\mathcal S}_2$,
for the map $(\hat{v}_1,\hat{v}_2) \mapsto \tilde{H}(\hat{v}_1,\hat{v}_2) $ from $\hat{\mathcal S}_1 \times \hat{\mathcal S}_2 \to 
2^{\hat{\mathcal S}_1} \times 2^{\hat{\mathcal S}_2}$, i.e., 
\[
(\hat{v}^*_1, \hat{v}^*_2) \in \tilde{H} (\hat{v}^*_1, \hat{v}^*_2) .
\]
This implies that $(\tilde{\psi}^{\hat{v}^*_2}_{\alpha, 1}, \tilde{\psi}^{\hat{v}^*_1}_{\alpha, 2})$ satisfies the following coupled HJB equations
\interdisplaylinepenalty=0
\begin{equation}\label{coupled-hjb1}
 \left\{\begin{aligned}
        \alpha \theta_1  \dfrac{d\tilde{\psi}^{\hat{v}^*_2}_{\alpha, 1}}{d\theta_1}(\theta_1,i)&= \inf_{v_1\in V_1}\Big [ \Pi_{v_1,\hat{v}_2^*(\theta_1,i)} 
\tilde{\psi}^{\hat{v}^*_2}_{\alpha, 1}(\theta_1,i) + \theta_1 r_1(i,v_1,\hat{v}_2^*(\theta_1,i))\tilde{\psi}^{\hat{v}^*_2}_{\alpha, 1}(\theta_1,i) \Big ] \, \,   \\
&=  \Pi_{\hat{v}_1^*(\theta_1,i),\hat{v}_2^*(\theta_1,i)} 
\tilde{\psi}^{\hat{v}^*_2}_{\alpha, 1}(\theta_1,i) + \theta_1 r_1(i,\hat{v}_1^*(\theta_1,i),\hat{v}_2^*(\theta_1,i))\tilde{\psi}^{\hat{v}^*_2}_{\alpha, 1}(\theta_1,i)   \\
\displaystyle{ \tilde{\psi}^{\hat{v}^*_2}_{\alpha, 1}(0,i) } &= 1 ,  \\
\alpha \theta_2  \dfrac{d\tilde{\psi}^{\hat{v}^*_1}_{\alpha, 2}}{d\theta_2}(\theta_2,i)&= \inf_{v_2\in V_2}\Big [ \Pi_{\hat{v}_1^*(\theta_2,i),v_2} 
\tilde{\psi}^{\hat{v}^*_1}_{\alpha, 2}(\theta_2,i) + \theta_2 r_2(i,\hat{v}_1^*(\theta_2,i),v_2)\tilde{\psi}^{\hat{v}^*_1}_{\alpha, 2}(\theta_2,i) \Big ] \, \,   \\
&=  \Pi_{\hat{v}_1^*(\theta_2,i),\hat{v}_2^*(\theta_2,i)} 
\tilde{\psi}^{\hat{v}^*_1}_{\alpha, 2}(\theta_2,i) + \theta_2 r_2(i,\hat{v}_1^*(\theta_2,i),\hat{v}_2^*(\theta_2,i))\tilde{\psi}^{\hat{v}^*_1}_{\alpha, 2}(\theta_2,i)   \\
\displaystyle{ \tilde{\psi}^{\hat{v}^*_1}_{\alpha, 2}(0,i) } &= 1 .
       \end{aligned}
 \right.
\end{equation}
Now from (\ref{valuefunctions}), we have
\begin{eqnarray*}
\tilde{\psi}^{\hat{v}^*_2}_{\alpha, 1} (\theta_1, i) & = & \inf_{v_1 \in {\mathcal M}_1} J^{v_1, \hat{v}^*_2}_{\alpha, 1}(\theta_1, i)\\
& = & J^{\hat{v}^*_1, \hat{v}^*_2}_{\alpha, 1}(\theta_1, i), \\
\tilde{\psi}^{\hat{v}^*_1}_{\alpha, 2} (\theta_2, i) & = & \inf_{v_2 \in {\mathcal M}_2} J^{\hat{v}^*_1, v_2}_{\alpha, 2}(\theta_2, i)\\
& = & J^{\hat{v}^*_1, \hat{v}^*_2}_{\alpha, 2}(\theta_2, i).
\end{eqnarray*}
Therefore we obtain
\begin{eqnarray*}
 J^{v_1, \hat{v}^*_2}_{\alpha, 1}(\theta_1, i) & \geq  & J^{\hat{v}^*_1, \hat{v}^*_2}_{\alpha, 1}(\theta_1, i), \ \forall \ 
v_1 \in {\mathcal M}_1 , \\
 J^{\hat{v}^*_1, v_2}_{\alpha, 2}(\theta_2, i) & \geq  & J^{\hat{v}^*_1, \hat{v}^*_2}_{\alpha, 2}(\theta_2, i), \ \forall \ 
v_2 \in {\mathcal M}_2.
\end{eqnarray*}
This proves the existence of $\alpha$-discounted Nash equilibrium which is eventually stationary Markov strategies. 
\end{proof}
\section{Vanishing Discount Asymptotics}\label{LF}
In this section  we prove the existence of Nash equilibrium strategies for the ergodic cost criterion in the class of  stationary Markov strategies under the following assumption:\\
  \noindent {\bf (A3)}(Lyapunov condition) There exist constants $b > 0, \ \delta >0$, a finite set $C$ and a map 
$W : S \to [1, \infty)$ with $W(i)\to \infty$ as $i \to \infty$, such that
\begin{eqnarray*}
\Pi_v W(i) &\leq&  -2 \delta W(i) + b I_C(i), \ i \in S, \ v \in V \, . 
 \end{eqnarray*} 
We refer to Guo and Hern$\acute{\rm a}$ndez-Lerma \cite{GuoLermabook} for examples of controlled continuous time Markov chains satisfying the above condition.

 Throughout this section, we assume that 
for every pair of stationary  Markov strategies  $(v_1,v_2)$ the corresponding Markov chain is irreducible. \\
%Under this assumption and (A2), it follows by the results of \cite{MT}, \cite{MT1} that the limit in (\ref{main1}) exists for stationary strategies. Thus ``lim sup" in (\ref{main1}) can be replaced by ``lim" for stationary strategies.\\
First we truncate our cost functions: This process plays a crucial role  in finding a Nash equilibrium of the game. For $k=1,2$, let $r_k^n : S \times  V \to [0, \ \infty)$ be given by
\begin{equation}
r_{k}^n:= \left\{
\begin{array}{ll}
r_k & {\rm if} \;  i\in \{1,2,\cdots,n\}\\
0 & \rm{otherwise}.   \\
\end{array} \right. \label{(R)}
\end{equation}
Then as in the previous section we can show the following result.
\begin{theorem}
Assume (A1). Then there exist a pair of Markov stationary strategy $(v_{1,n}^*,v_{2,n}^*)$ and a pair of  bounded, continuously differentiable functions $(\psi^{v_{2,n}^*}_{\alpha,1n},\psi^{v_{1,n}^*}_{\alpha,2n})$ such that 
\interdisplaylinepenalty=0
 \begin{equation}\label{coupled-hjb1}
 \left\{\begin{aligned}
\alpha \theta_1  \dfrac{d\psi^{v_{2,n}^*}_{\alpha,1n}}{d\theta_1}(\theta_1,i)&= \inf_{v_1\in V_1}\Big [ \Pi_{v_1,v_{2,n}^*(i)} 
\psi^{v_{2,n}^*}_{\alpha,1n}(\theta_1,i) + \theta_1 r_1^n(i,v_1,v_{2,n}^*(i))\psi^{v_{2,n}^*}_{\alpha,1n}(\theta_1,i) \Big ] \, \,  \\
&=  \Pi_{v_{1,n}^*,v_{2,n}^*(i)} 
\psi^{v_{2,n}^*}_{\alpha,1n}(\theta_1,i) + \theta_1 r_1^n(i,v_{1,n}^*,v_{2,n}^*(i))\psi^{v_{2,n}^*}_{\alpha,1n}(\theta_1,i),   \\
\displaystyle{ \psi^{v_{2,n}^*}_{\alpha,1 }(0,i) } &= 1 ,  \\
\alpha \theta_2  \dfrac{d\psi^{v_{1,n}^*}_{\alpha,2n}}{d\theta_2}(\theta_2,i)&= \inf_{v_2\in V_2}\Big [ \Pi_{v_{1,n}^*(i),v_2} 
\psi^{v_{1,n}^*}_{\alpha,2n}(\theta_2,i) + \theta_2 r_2^n(i,v_{1,n}^*(i),v_2)\psi^{v_{1,n}^*}_{\alpha,2n}(\theta_2,i) \Big ] \, \,   \\
&=  \Pi_{v_{1,n}^*(i),v_{2,n}^*} 
\psi^{v_{1,n}^*}_{\alpha,2n}(\theta_2,i) + \theta_2 r_2^n(i,v_{1,n}^*(i),v_{1,n}^*)\psi^{v_{1,n}^*}_{\alpha,2n}(\theta_2,i),   \\
\displaystyle{ \psi^{v_{1,n}^*}_{\alpha,2}(0,i) } &= 1 .
 \end{aligned}
 \right.
\end{equation}
\end{theorem}

Let 
\begin{equation}\label{phidefinition}
 \phi^{v_{2,n}^*}_{\alpha,1n}(\theta_1,i):=\frac{1}{\theta_1} \ln \psi^{v_{2,n}^*}_{\alpha,1n}(\theta_1,i) .
\end{equation}
Then 
$$
\alpha\theta_1\frac{d \psi^{v_{2,n}^*}_{\alpha,1n}}{d \theta_1} =\theta_1 \left( \alpha \phi^{v_{2,n}^*}_{\alpha,1n}
+\theta_1 \alpha\dfrac{d \phi^{v_{2,n}^*}_{\alpha,1n}}{d \theta_1}\right) \psi^{v_{2,n}^*}_{\alpha,1n} .
$$
Let $\theta_1 \left( \alpha \phi^{v_{2,n}^*}_{\alpha,1n}
+\theta_1 \alpha\dfrac{d \phi^{v_{2,n}^*}_{\alpha,1n}}{d \theta_1}\right)=g^{v_{2,n}^*}_{\alpha,1n}(\theta_1,i)$. Then 
 $\psi^{v_{2,n}^*}_{\alpha,1n}$ is a solution to the ODE
 \interdisplaylinepenalty=0
 \begin{equation*}
 \left\{\begin{aligned}
0&= \displaystyle{\inf_{v_1\in V_1}  \Big[ \Pi_{v_1,v_{2,n}^*(i)} \psi^{v_{2,n}^*}_{\alpha,1n} +(\theta_1 r_1^n(i,v_1,v_{2,n}^*(i))-g^{v_{2,n}^*}_{\alpha,1n}(\theta_1,i))
\psi^{v_{2,n}^*}_{\alpha,1n}\Big ] }\\
\displaystyle{\psi^{v_{2,n}^*}_{\alpha,1n}(0,i) } &= 1 .
 \end{aligned}
 \right.
\end{equation*}
In what follows we fix a reference state $i_0 \in S$ satisfying
$$W(i_0)\geq 1+\frac{b}{\delta}.$$
Set $$\bar{\psi}^{v_{2,n}^*}_{\alpha,1n}(\theta_1,i)=\dfrac{\psi^{v_{2,n}^*}_{\alpha ,1n}(\theta_1,i)}{\psi^{v_{2,n}^*}_{\alpha ,1n}(\theta_1,i_0)}.$$
Then a straightforward calculation shows that $\bar{\psi}^{v_{2,n}^*}_{\alpha,1n}$
is a solution to the ODE
\interdisplaylinepenalty=0
 \begin{equation*}
 \left\{\begin{aligned}
0&= \displaystyle{\inf_{v_1\in V_1}  \Big[ \Pi_{v_1,v_{2,n}^*(i)} \bar{\psi}^{v_{2,n}^*}_{\alpha,1n} +(\theta_1 r_1^n(i,v_1,v_{2,n}^*(i))-g^{v_{2,n}^*}_{\alpha,1n}(\theta_1,i))
\bar{\psi}^{v_{2,n}^*}_{\alpha,1n}\Big ] }\\
\displaystyle{ \bar{\psi}^{v_{2,n}^*}_{\alpha,1n}(\theta_1,i_0) } &= 1 .
 \end{aligned}
 \right.
\end{equation*} 
Using analogous arguments we can show that 
$\bar{\psi}^{v_{1,n}^*}_{\alpha,2n}$
is a solution to the ODE
\interdisplaylinepenalty=0
\begin{equation*}
 \left\{\begin{aligned}
0&= \displaystyle{\inf_{v_2\in V_2}  \Big[ \Pi_{v_{1,n}^*(i),v_2} \bar{\psi}^{v_{1,n}^*}_{\alpha,2n} +(\theta_2 r_2^n(i,v_{1,n}^*(i),v_2)-g^{v_{1,n}^*}_{\alpha,2n}(\theta_2,i))
\bar{\psi}^{v_{1,n}^*}_{\alpha,2n}\Big ] }\\
\displaystyle{ \bar{\psi}^{v_{1,n}^*}_{\alpha,2n}(\theta_2,i_0) } &= 1 ,
 \end{aligned}
 \right.
\end{equation*} 
where  $\theta_2 \left( \alpha \phi^{v_{1,n}^*}_{\alpha,2n}
+\theta_2 \alpha\dfrac{d \phi^{v_{1,n}^*}_{\alpha,2n}}{d \theta_2}\right)=g^{v_{1,n}^*}_{\alpha,2n}(\theta_2,i)$.

This immediately yields the following result:
\begin{theorem}\label{thm17}
Assume (A1). Then there exist a pair of Markov stationary strategy $(v_{1,n}^*,v_{2,n}^*)$ and a pair of bounded, continuously differentiable functions $( \bar{\psi}^{v_{2,n}^*}_{\alpha,1n}, \bar{\psi}^{v_{1,n}^*}_{\alpha,2n})$ such that 
\interdisplaylinepenalty=0
\begin{equation}\label{coupled-hjb2}
 \left\{\begin{aligned}
0&= \displaystyle{\inf_{v_1\in V_1}  \Big[ \Pi_{v_1,v_{2,n}^*(i)} \bar{\psi}^{v_{2,n}^*}_{\alpha,1n} +(\theta_1 r_1^n(i,v_1,v_{2,n}^*(i))-g^{v_{2,n}^*}_{\alpha,1n}(\theta_1,i))
\bar{\psi}^{v_{2,n}^*}_{\alpha,1n}\Big ] }\, \,   \\
\displaystyle{ \bar{\psi}^{v_{2,n}^*}_{\alpha,1n}(\theta_1,i_0) } &= 1, \\
0&= \displaystyle{\inf_{v_2\in V_2}  \Big[ \Pi_{v_{1,n}^*(i),v_2} \bar{\psi}^{v_{1,n}^*}_{\alpha,2n} +(\theta_2 r_2^n(i,v_{1,n}^*(i),v_2)-g^{v_{1,n}^*}_{\alpha,2n}(\theta_2,i))
\bar{\psi}^{v_{1,n}^*}_{\alpha,2n}\Big ] }\, \,  \\
\displaystyle{ \bar{\psi}^{v_{1,n}^*}_{\alpha,2n}(\theta_2,i_0) } &= 1 .
 \end{aligned}
 \right.
\end{equation}
\end{theorem}
Next we want to take limit $\alpha \to 0$. To this end we show that $\bar{\psi}^{v_{2,n}^*}_{\alpha,1n}(\theta_1,i),\;\bar{\psi}^{v_{1,n}^*}_{\alpha,2n}(\theta_2,i)  $ and $g^{v_{2,n}^*}_{\alpha,1n}(\theta_1,i),\; g^{v_{1,n}^*}_{\alpha,2n}(\theta_2,i)  $ are uniformly bounded in $\alpha$ for each $i \in S $. %By closely mimicking the arguments in [\cite{ABS}, Lemma 2.1.], one can derive  the following bound; we omit the details.
\begin{lemma}\label{alpha1}
Assume (A1). Let $\phi^{v_{2,n}^*}_{\alpha,1n}, \; \phi^{v_{1,n}^*}_{\alpha,2n}$ be given by (\ref{phidefinition}), then the following inequalities hold:
 $$
\| \alpha \phi^{v_{2,n}^*}_{\alpha,1n}\|_{\infty} + \Big \|\alpha\theta_1\frac{d \phi^{v_{2,n}^*}_{\alpha,1n}}{d \theta_1} \Big \|_{\infty}
\leq 3 \|r_1 \|_{\infty},  \ \forall\; 0<\alpha<1,\ 0<\theta < \Theta.
$$
and 
 $$
\| \alpha \phi^{v_{1,n}^*}_{\alpha,2n}\|_{\infty} + \Big \|\alpha\theta_2\frac{d \phi^{v_{1,n}^*}_{\alpha,2n}}{d \theta_2} \Big \|_{\infty}
\leq 3 \|r_2 \|_{\infty},  \ \forall\; 0<\alpha<1,\ 0<\theta < \Theta.
$$
%where $\|\cdot\|_{\infty}$ denote the supnorm.
\end{lemma}
\begin{proof} 
From Lemma \ref{alphabound}, we have 
$$
1 \leq \psi^{v_{2,n}^*}_{\alpha,1n}(\theta_1,i) \leq e^{\frac{\theta_1 \|r_1\|_{\infty} }{\alpha} },
\ \  \forall\; 0<\theta_1 < \Theta, i \in S ,\
0<\alpha<1.
$$
Therefore
$$
\| \alpha \phi^{v_{2,n}^*}_{\alpha,1n}\|_{\infty} \leq  \|r_1\|_{\infty},
\ \  \forall\; 0<\theta_1 < \Theta, i \in S ,\
0<\alpha<1.
$$
For $(v_1,v_{2,n}^*)\in \mathcal{M}_1 \times \mathcal{S}_2, \, i\in S $, set 
 $$
F_{\alpha}^{1n}(i,v_1,v_{2,n}^*)=\int_{0}^{\infty} e^{-\alpha t} \;r_1^n(Y(t-), v_1(t,Y(t-)),v_{2,n}^*(Y(t-))) dt
$$ 
and 
$$
H_{\alpha}^{1n}(\theta_1,i,v_1,v_{2,n}^*)= \log E_i^{v_1,v_{2,n}^*}\left[ e^{\theta_1 \int_{0}^{\infty} e^{-\alpha t} \;r_1^n(Y(t-), v_1(t,Y(t-)),v_{2,n}^*(Y(t-))) dt}\right].
$$
 Then we have
$$
 \frac{d H_{\alpha}^{1n} }{ d \theta_1} =\dfrac{1}{E_i^{v_1,v_{2,n}^*}[ e^{\theta_1 F_{\alpha}^{1n}}]} E_i^{v_1,v_{2,n}^*}[F_{\alpha}^{1n} e^{\theta_1 F_{\alpha}^{1n}}] \leq
\frac{ \|r_1\|_{\infty}}{\alpha}.$$
 Using analogous arguments as in the proof of Lemma \ref{alphabound}, we can show that for each $\epsilon >0 $, 
%$$
%H_{\alpha}^{1n}(\theta_1+\epsilon,i,v_1,v_{2,n}^*)-H_{\alpha}^{1n}(\theta_1,i,v_1,v_{2,n}^*)= \epsilon \frac{d H_{\alpha}^{1n}}{d \theta_1}
 %   (\theta_{\epsilon},i,v_1,v_{2,n}^*),
%$$
%for some $\theta_{\epsilon} \in (\theta_1,\theta_1+\epsilon)$.
%Therefore 
%$$
%|H_{\alpha}^{1n}(\theta_1+\epsilon,i,v_1,v_{2,n}^*)-H_{\alpha}^{1n}(\theta_1,i,v_1,v_{2,n}^*)| \leq  \epsilon \frac{ \|r_1\|_{\infty}}{\alpha}.
%$$
%Note that
\begin{eqnarray*}
|(\theta_1+\epsilon) \phi^{v_{2,n}^*}_{\alpha,1n}(\theta_1+\epsilon,i)-\theta_1 \phi^{v_{2,n}^*}_{\alpha,1n}(\theta_1,i)| 
 &\leq & \sup_{v_1 \in \mathcal{M}_1} 
|H_{\alpha}^{1n}(\theta_1+\epsilon,i,v_1,v_{2,n}^*)-H_{\alpha}^{1n}(\theta_1,i,v_1,v_{2,n}^*)| \\
&\leq &  \epsilon \frac{ \|r_1\|_{\infty}}{\alpha}.
\end{eqnarray*}
Analogous bound can be obtained for $\epsilon<0$. Therefore we have 
\begin{eqnarray*}
 \Big \|\alpha \frac{d (\theta_1 \phi^{v_{2,n}^*}_{\alpha,1n})}{d \theta_1} \Big \|_{\infty}
\leq  \|r_1\|_{\infty}.
\end{eqnarray*}
Note that 
\begin{eqnarray*}
\Big \|\alpha \theta_1 \frac{d  \phi^{v_{2,n}^*}_{\alpha,1n}}{d \theta_1}  \Big \|_{\infty} \leq \Big \|\alpha \frac{d (\theta_1 \phi^{v_{2,n}^*}_{\alpha,1n})}{d \theta_1} \Big \|_{\infty}+\| \alpha \phi^{v_{2,n}^*}_{\alpha,1n}\|_{\infty}.
\end{eqnarray*}
Hence
 $$
\| \alpha \phi^{v_{2,n}^*}_{\alpha,1n}\|_{\infty} + \Big \|\alpha\theta_1\frac{d \phi^{v_{2,n}^*}_{\alpha,1n}}{d \theta_1} \Big \|_{\infty}
\leq 3 \|r_1 \|_{\infty},  \ \forall\; 0<\alpha<1,\ 0<\theta < \Theta.
$$
Using analogous arguments we can show that 
$$
\| \alpha \phi^{v_{1,n}^*}_{\alpha,2n}\|_{\infty} + \Big \|\alpha\theta_2\frac{d \phi^{v_{1,n}^*}_{\alpha,2n}}{d \theta_2} \Big \|_{\infty}
\leq 3 \|r_2 \|_{\infty},  \ \forall\; 0<\alpha<1,\ 0<\theta < \Theta.
$$
This completes the proof.
\end{proof}

Form the Lemma \ref{alpha1}, it is clear that $g^{v_{2,n}^*}_{\alpha,1n}(\theta_1,i),\; g^{v_{1,n}^*}_{\alpha,2n}(\theta_2,i)  $ are uniformly bounded. 
For each $i \in S $, we want to show that $\bar{\psi}^{v_{2,n}^*}_{\alpha,1n}(\theta_1,i),\;\bar{\psi}^{v_{1,n}^*}_{\alpha,2n}(\theta_2,i)  $ are uniformly (in $\alpha$) bounded.\\
We need the following result which is proved by Kumar and Pal [\cite{SP}, Theorem 3.1].
\begin{theorem}\label{thmdpp}
 Assume (A1). For any set $\tilde{S} \subseteq S,$ 
\[
\psi^{v_{2,n}^*}_{\alpha,1n}(\theta_1, i) \ \leq \, \inf_{v_1 \in {\mathcal M_1}} E_i^{v_1,v_{2,n}^*} \Big[ e^{\theta_1 \int^{ \tau}_0 e^{-\alpha s} r_1^n(Y(s-), v_1(s,Y(s-)),v_{2,n}^*(Y(s-))) ds } 
\psi^{v_{2,n}^*}_{\alpha,1n}(\theta_1 e^{-\alpha \tau}, Y( \tau)) \Big] , 
\]
where $\tau$ is the hitting time of the process $Y(t)$ corresponding to $(v_1,v_{2,n}^*) \in {\mathcal M}_1 \times \mathcal{M}_2$ to the set $\tilde{S}$. 
\end{theorem}
%Using the similar arguments as in the proofs of \cite{SP1}, we can prove the following results; we omit the details.
\begin{lemma}\label{expbd}
 Assume (A1) and (A3). Let $Y(t)$ be the process (\ref{cmc2}) corresponding to $(v_1,v_2) \in {\mathcal M}_1 \times \mathcal{M}_2$ and let $C_0=\{j\in S: W(j)\geq 1+\frac{b}{\delta}\}$.  Then for each $j \in C_0$  
\begin{eqnarray*}
 E_{i}^{v_1,v_2} \left [e^{\delta \tau_{j}}  \right] &\leq &  W(i)  ,
\end{eqnarray*}
where $\tau_{j} =\inf \{t\geq 0:Y(t)=j\}$.
\end{lemma}
\begin{proof}
Applying It$\hat{{\rm o}}$-Dynkin's formula to $f(t)=e^{\delta t}W(Y(t))$, and using (A3), we obtain
 \interdisplaylinepenalty=0
\begin{eqnarray*}
&& E_{i}^{v_1,v_2}[f(\tau_{j}\wedge \tau_{N})-W(i)]\\
&=&E_{i}^{v_1,v_2}\Big [\int_{0}^{\tau_{j}\wedge \tau_{N}}e^{\delta s}
[\Pi_{v_1(s,Y(s-)),v_2(s,Y(s-))} W(Y(s) )+\delta W(Y(s))] ds \Big]\nonumber   \\
&\leq & E_{i}^{v_1,v_2}\Big[\int_{0}^{\tau_{j} \wedge \tau_{N}}e^{\delta s} (- \delta W(Y(s)) + b I_C(Y(s))) ds\Big]
\nonumber \\ 
&\leq&  (\frac{b}{\delta}-1) E_{i}^{v_1,v_2}\Big[e^{\delta (\tau_{j} \wedge \tau_{N})} \Big] \leq \frac{b}{\delta} E_{i}^{v_1,v_2}\Big[e^{\delta \tau_{j}} \Big], 
\end{eqnarray*}
where $\tau_N \, = \, \inf\{ t \geq 0 | Y(t) \notin \{1, 2, \dots, N\} \}$. Hence
\[
E_{i}^{v_1,v_2}[e^{\delta(\tau_{j}\wedge \tau_{N})}W(Y(\tau_{j}\wedge \tau_{N}))]
\ \leq \ W(i)+\frac{b}{\delta} E_{i}^{v_1,v_2}\Big[e^{\delta \tau_{j}} \Big] \, .
\]
By letting $N\rightarrow \infty$ and invoking Fatou's lemma, we obtain
\begin{eqnarray*}
 E_{i}^{v_1,v_2} \left [e^{\delta \tau_{j}}W(Y(\tau_j))  \right] &\leq &  W(i)+\frac{b}{\delta} E_{i}^{v_1,v_2}\Big[e^{\delta \tau_{j}} \Big]  .
\end{eqnarray*}
Therefore,
\[
 E_{i}^{v_1,v_2}[e^{\delta \tau_{j}}(W(j)-\frac{b}{\delta})] \ \leq \  W(i)  \, .
\]
Since $W(j)\geq 1+\frac{b}{\delta}$ , we get
\begin{eqnarray*}
 E_{i}^{v_1,v_2} \left [e^{\delta \tau_{j}}  \right] &\leq &  W(i).
\end{eqnarray*}
This completes the proof.
\end{proof}

Before proceeding further we make the following small cost assumption.\\
\noindent {\bf (A4)} $\theta_1 \|r_1\|_{\infty} \leq \delta$ and $\theta_2 \|r_2\|_{\infty} \leq \delta$,  where $\delta >0$ is as in (A3).
\begin{lemma}\label{expbd3}
 Assume (A1), (A3) and (A4). Let $Y(t)$ be the process (\ref{cmc2}) corresponding to $(v_1,v_2) \in {\mathcal M}_1 \times \mathcal{M}_2$.  Then for $k=1,2$, we have
\begin{eqnarray*}
&&E_{i}^{v_1,v_2}[e^{\theta_k \int_0^{T} r_k(Y(t-), v_1(t,Y(t-)),v_{2}(t,Y(t-)))ds}W(Y(T))] \\
\ &\leq& \ (W(i)+ b T) E_{i}^{v_1,v_2}\Big[e^{\theta_k \int_0^T r_k(Y(t-), v_1(t,Y(t-)),v_{2}(t,Y(t-)))ds} \Big],\,T \in [0,\infty).
\end{eqnarray*}
\end{lemma}
\begin{proof}
  Applying It$\hat{{\rm o}}$-Dynkin's formula to $$g(t)=e^{\theta_k \int_0^t r_k(Y(s-), v_1(s,Y(s-)),v_{2}(s,Y(s-)))ds}W(Y(t)),$$ and using (A3) and (A4), we obtain
  \interdisplaylinepenalty=0
\begin{eqnarray*}
&&E_{i}^{v_1,v_2}[g(T\wedge \tau_{N})-W(i)] \\
&=&E_{i}^{v_1,v_2}\Big[\int_{0}^{T\wedge  \tau_{N}}e^{\theta_k \int_0^s r_k(Y(t-), v_1(t,Y(t-)),v_{2}(t,Y(t-)))dt}
\Big[\Pi_{v_1(s,Y(s-)),v_2(s,Y(s-))} W(Y(s) )\\
&+&\theta_k r_k(Y(s-), v_1(s,Y(s-)),v_{2}(s,Y(s-))) W(Y(s)) \Big ] ds \Big]\nonumber   \\
&\leq &E_{i}^{v_1,v_2}\Big[\int_{0}^{T\wedge  \tau_{N}}e^{\theta_k \int_0^s r_k(Y(t-), v_1(t,Y(t-)),v_{2}(t,Y(t-)))dt}\\
&&\Big[\Pi_{v_1(s,Y(s-)),v_2(s,Y(s-))} W(Y(s) )+ \delta W(Y(s)) \Big ]  ds \Big ] \nonumber   \\
&\leq & E_{i}^{v_1,v_2}\Big[\int_{0}^{T\wedge \tau_{N}}e^{\theta_k \int_0^s r_k(Y(t-), v_1(t,Y(t-)),v_{2}(t,Y(t-)))dt} (- \delta W(Y(s)) + b I_C(Y(s))) ds\Big]
\nonumber \\ 
&\leq & b T E_{i}^{v_1,v_2}\Big[e^{\theta_k \int_0^T r_k(Y(t-), v_1(t,Y(t-)),v_{2}(t,Y(t-)))dt} \Big],
\end{eqnarray*}
where $\tau_N \, = \, \inf\{ t \geq 0 | Y(t) \notin \{1, 2, \dots, N\} \}$. Hence
 \interdisplaylinepenalty=0
\begin{eqnarray*}
&&E_{i}^{v_1,v_2}[e^{\theta_k \int_0^{T\wedge \tau_{N}} r_k(Y(t-), v_1(t,Y(t-)),v_{2}(t,Y(t-)))dt}W(Y(T\wedge \tau_{N}))] \\
\ &\leq& \ W(i)+ b T E_{i}^{v_1,v_2}\Big[e^{\theta_k \int_0^T r_k(Y(t-), v_1(t,Y(t-)),v_{2}(t,Y(t-)))dt} \Big]\, .
\end{eqnarray*}
By letting $N\rightarrow \infty$ and invoking Fatou's lemma, we obtain
\begin{eqnarray*}
&&E_{i}^{v_1,v_2}[e^{\theta_k \int_0^{T} r_k(Y(t-), v_1(t,Y(t-)),v_{2}(t,Y(t-)))dt}W(Y(T))] \\
\ &\leq& \ (W(i)+ b T) E_{i}^{v_1,v_2}\Big[e^{\theta_k \int_0^T r_k(Y(t-), v_1(t,Y(t-)),v_{2}(t,Y(t-)))dt} \Big]\, .
\end{eqnarray*}
This completes the proof. 
\end{proof}
\begin{lemma}\label{bd1} Assume (A1), (A3) and (A4). 
Then $$\bar{\psi}^{v_{2,n}^*}_{\alpha ,1n}(\theta_1, i) \, \leq \, W(i) , \; 
i \in S$$ 
and $$ \bar{\psi}^{v_{1,n}^*}_{\alpha ,2n}(\theta_2, i) \, \leq \, W(i) , \; 
i \in S. $$
\end{lemma}
\begin{proof} From Theorem \ref{thmdpp}, we have for $\hat{v}_1\in \mathcal{M}_1$,
\interdisplaylinepenalty=0
\begin{eqnarray*}
&& \psi^{v_{2,n}^*}_{\alpha ,1n}(\theta_1, i) \\ \ &\leq & \,  E_i^{\hat{v}_1,v_{2,n}^*} \Big[ e^{\theta_1 \int^{ \tau_{i_0}}_0 e^{-\alpha s} r_1(Y(s-), \hat{v}_1((s,Y(s-))),v_{2,n}^*(Y(s-))) ds } 
\psi^{v_{2,n}^*}_{\alpha ,1n}(\theta_1 e^{-\alpha \tau}, Y( \tau_{i_0})) \Big]\\
& = &  \,  E_i^{\hat{v}_1,v_{2,n}^*} \Big[ e^{\theta_1 \int^{ \tau_{i_0}}_0 e^{-\alpha s} r_1(Y(s-), \hat{v}_1((s,Y(s-))),v_{2,n}^*(Y(s-))) ds } 
\psi^{v_{2,n}^*}_{\alpha ,1n}(\theta_1 e^{-\alpha \tau}, i_0)) \Big]\\
&\leq &  \,  E_i^{\hat{v}_1,v_{2,n}^*} \Big[ e^{\theta_1 \int^{ \tau_{i_0}}_0 e^{-\alpha s} r_1(Y(s-), \hat{v}_1((s,Y(s-))),v_{2,n}^*(Y(s-))) ds } 
\psi^{v_{2,n}^*}_{\alpha ,1n}(\theta_1, i_0) \Big]\\
\end{eqnarray*}
where $\tau_{i_0} \, = \, \inf \{ t \geq 0 | Y(t) = i_0\}$. In the last inequality we used the fact that
$\psi^{v_{2,n}^*}_{\alpha ,1n}(\cdot, i)$ is nondecreasing in $\theta_1$ for each fixed $i$.
Hence
\interdisplaylinepenalty=0
\[
\begin{array}{lll}
&& \bar{\psi}^{v_{2,n}^*}_{\alpha ,1n}(\theta_1, i) \\ \ &\leq & \ E_i^{\hat{v}_1,v_{2,n}^*} \Big[ e^{\theta_1 \int^{ \tau_{i_0}}_0 e^{-\alpha s} r_1(Y(s-), \hat{v}_1((s,Y(s-))),v_{2,n}^*(Y(s-))) ds }  \Big] \\
\ &\leq &\  E_i^{\hat{v}_1,v_{2,n}^*} e^{\theta_1 \|r_1\|_{\infty} \tau_{i_0}} \leq \, W(i)  \, .
\end{array}
\]
The last inequality follows from  Lemma \ref{expbd}. Using analogous arguments  we can show that $$ \bar{\psi}^{v_{1,n}^*}_{\alpha ,2n}(\theta_2, i) \, \leq \, W(i) , \; 
i \in S. $$
This completes the proof. 
\end{proof}
\begin{lemma}\label{ubd} Assume (A1), (A3) and (A4).
Then $$\sup_{\alpha>0, i\in S} \bar{\psi}^{v_{2,n}^*}_{\alpha ,1n}(\theta_1, i) < \infty$$ and  $$\sup_{\alpha>0, i\in S} \bar{\psi}^{v_{1,n}^*}_{\alpha ,2n}(\theta_2, i) < \infty.$$
\end{lemma}
\begin{proof}
Let $i\geq n+1$ and let $Y(t)$ be the solution corresponding to 
 $(v_1,v_{2,n}^*) \in {\mathcal M}_1 \times \mathcal{S}_2$ with initial condition $i$. Then from Theorem \ref{thmdpp}, we have 
 \interdisplaylinepenalty=0
\[
\begin{array}{lll}
&& \psi^{v_{2,n}^*}_{\alpha ,1n}(\theta_1, i) \\ \ &\leq & \,\displaystyle{ E_i^{v_1,v_{2,n}^*} \Big[ e^{\theta_1 \int^{ \tau}_0 e^{-\alpha s} r_1^n(Y(s-), v_1(s,Y(s-)),v_{2,n}^*(Y(s-))) ds } 
\psi^{v_{2,n}^*}_{\alpha ,1n}(\theta_1 e^{-\alpha \tau}, Y( \tau)) \Big]}\\
&\leq & \,\displaystyle{  E_i^{v_1,v_{2,n}^*} \Big[  
\psi^{v_{2,n}^*}_{\alpha ,1n}(\theta_1 e^{-\alpha \tau}, Y( \tau)) \Big]}\\
&\leq & \,\displaystyle{  E_i^{v_1,v_{2,n}^*} \Big[  
\psi^{v_{2,n}^*}_{\alpha ,1n}(\theta_1, Y( \tau)) \Big]}\\
\end{array}
\]
where
\[
\tau =\inf \{t\geq 0:Y(t)\in \{1,2,\cdots ,n\}\} . 
\]

 In the last inequality we used the fact that
$\psi^{v_{2,n}^*}_{\alpha ,1n}(\cdot, i)$ is nondecreasing in $\theta_1$ for each fixed $i$.
Hence
\interdisplaylinepenalty=0
\[
\bar{\psi}^{v_{2,n}^*}_{\alpha ,1n}(\theta_1, i) \ \leq \ 1+ \max_{j =1, \dots , n } 
\bar{\psi}^{v_{2,n}^*}_{\alpha ,1n}(\theta_1, j) \, 
 \leq \, 1+  \max_{j =1, \dots , n }  W(j)  \, .
\]
since  $\psi^{v_{2,n}^*}_{\alpha ,1n}(\theta_1, i_0) \geq 1$ and last inequality follows from  Lemma \ref{bd1}. Therefore for each $n \geq 1, \bar{\psi}^{v_{2,n}^*}_{\alpha ,1n}$ is  bounded. Similarly we can prove that for each $n \geq 1, \bar{\psi}^{v_{1,n}^*}_{\alpha ,2n}$ is  bounded. This completes the proof.
\end{proof}
\begin{lemma}\label{bd2} Assume (A1), (A3) and  (A4). 
Then $$\bar{\psi}^{v_{2,n}^*}_{\alpha ,1n}(\theta_1, i) \, \geq \, \dfrac{1} {W(i_0) }, \;  
i \in C_0 $$ 
 and 
$$ \bar{\psi}^{v_{1,n}^*}_{\alpha ,2n}(\theta_2, i) \, \geq \, \dfrac{1} {W(i_0) },  \; i \in C_0,$$
where $C_0$ is as in Lemma \ref{expbd}.
\end{lemma}
\begin{proof} From Theorem \ref{thmdpp}, we have for $\hat{v}_1\in \mathcal{M}_1$,
\interdisplaylinepenalty=0
\[
\begin{array}{lll}
&& \psi^{v_{2,n}^*}_{\alpha ,1n}(\theta_1, i_0)\\ \ &\leq & \,  E_{i_0}^{\hat{v}_1,v_{2,n}^*} \Big[ e^{\theta_1 \int^{ \tau_{i}}_0 e^{-\alpha s} r_1(Y(s-), \hat{v}_1((s,Y(s-))),v_{2,n}^*(Y(s-))) ds } 
\psi^{v_{2,n}^*}_{\alpha ,1n}(\theta_1 e^{-\alpha \tau}, Y( \tau_{i})) \Big]\\
& = &  \,  E_{i_0}^{\hat{v}_1,v_{2,n}^*} \Big[ e^{\theta_1 \int^{ \tau_{i}}_0 e^{-\alpha s} r_1(Y(s-), \hat{v}_1((s,Y(s-))),v_{2,n}^*(Y(s-))) ds } 
\psi^{v_{2,n}^*}_{\alpha ,1n}(\theta_1 e^{-\alpha \tau}, i)) \Big]\\
&\leq &  \,  E_{i_0}^{\hat{v}_1,v_{2,n}^*} \Big[ e^{\theta_1 \int^{ \tau_{i}}_0 e^{-\alpha s} r_1(Y(s-), \hat{v}_1((s,Y(s-))),v_{2,n}^*(Y(s-))) ds } 
\psi^{v_{2,n}^*}_{\alpha ,1n}(\theta_1, i) \Big]\\
\end{array}
\]
where $\tau_{i} \, = \, \inf \{ t \geq 0 | Y(t) = i\}$. In the last inequality we used the fact that
$\psi^{v_{2,n}^*}_{\alpha ,1n}(\cdot, i)$ is nondecreasing in $\theta_1$ .
Hence for $i \in C_0$
\interdisplaylinepenalty=0
\[
\begin{array}{lll}
\dfrac{1}{\bar{\psi}^{v_{2,n}^*}_{\alpha ,1n}(\theta_1, i)} \ &\leq & \ E_{i_0}^{\hat{v}_1,v_{2,n}^*} \Big[ e^{\theta_1 \int^{ \tau_{i}}_0 e^{-\alpha s} r_1(Y(s-), \hat{v}_1((s,Y(s-))),v_{2,n}^*(Y(s-))) ds }  \Big] \\
\ &\leq &\  E_{i_0}^{\hat{v}_1,v_{2,n}^*} e^{\theta_1 \|r_1\|_{\infty} \tau_{i}} \leq \,  W(i_0)  \, .
\end{array}
\]
The last inequality follows from  Lemma \ref{expbd}.  Using analogous arguments  we obtain the other bound. This completes the proof. \end{proof}

\begin{theorem}\label{bex17}
 Assume (A1), (A3) and  (A4). Then there exist a pair of stationary Markov strategies $(v_{1,n}^*,v_{2,n}^*) $, a pair of scalars   $(\rho^{*}_{1n},\rho^{*}_{2n})$ and  a pair of functions $(\hat{\psi}^{*}_{1n}(i)), \hat{\psi}^{*}_{2n}(i))$ in $B_{W}(S)\times B_{W}(S)$ such that
 \interdisplaylinepenalty=0
 \begin{equation}\label{brs17}
 \left\{\begin{aligned}
 \theta_1 \rho^{*}_{1n} ~\hat{\psi}^{*}_{1n}(i) &= \inf_{v_1\in V_1} \Big [ \Pi_{v_1,v_{2,n}^*(i)} 
\hat\psi^{*}_{1n}(i) +\theta_1 r_1^n(i,v_1,v_{2,n}^*(i))\hat\psi^{*}_{1n}(i) \Big ]   \\
&=  \Pi_{v_{1,n}^*(i),v_{2,n}^*(i)} 
\hat\psi^{*}_{1n}(i) +\theta_1 r_1^n(i,v_{1,n}^*(i),v_{2,n}^*(i))\hat\psi^{*}_{1n}(i)    \\
\displaystyle{ \hat\psi^{*}_{1n}(i_0) } &= 1, \\
 \theta_2 \rho^{*}_{2n} ~\hat{\psi}^{*}_{2n}(i) &= \inf_{v_2\in V_2} \Big [ \Pi_{v_{1,n}^*(i),v_2} 
\hat\psi^{*}_{2n}(i) +\theta_2 r_2^n(i,v_{1,n}^*(i),v_{2})\hat\psi^{*}_{2n}(i) \Big ]  \\
&=  \Pi_{v_{1,n}^*(i),v_{2,n}^*(i)} 
\hat\psi^{*}_{2n}(i) +\theta_2 r_2^n(i,v_{1,n}^*(i),v_{2,n}^*(i))\hat\psi^{*}_{2n}(i)    \\
\displaystyle{ \hat\psi^{*}_{2n}(i_0) } &= 1. 
 \end{aligned}
 \right.
\end{equation}
Moreover, $\displaystyle{\sup_{n}\{\rho^{*}_{1n},\rho^{*}_{2n} \}}\leq \delta$.
\end{theorem}
\begin{proof}
Using Lemma \ref{bd1}, for each fix  $i$, 
 $ \{ \bar{\psi}^{v_{2,n}^*}_{\alpha ,1n}(\theta_1,i)|  \alpha >0\}$ and $ \{ \bar{\psi}^{v_{1,n}^*}_{\alpha ,2n}(\theta_2,i)|  \alpha >0\}$ are bounded. Hence along a subsequence, denoted by the same notation with an abuse of notation, we have
$\bar{\psi}^{v_{2,n}^*}_{\alpha ,1n}(\theta_1,i) \rightarrow \hat{\psi}^{*}_{1n}(i)$
and $\bar{\psi}^{v_{1,n}^*}_{\alpha ,2n}(\theta_2,i) \rightarrow \hat{\psi}^{*}_{2n}(i)$ for each $i \in S$ 
for some $\hat{\psi}^{*}_{kn}: S \to [0,\infty), k=1,2$ as  $\alpha \rightarrow 0$. By Lemma \ref{bd1}, $\hat{\psi}^{*}_{kn}\in B_W(S)$ for $k=1,2$. By Lemma \ref{alpha1} it follows that along a further subsequence 
\begin{equation}\label{rho1}
 \alpha \phi^{v_{2,n}^*}_{\alpha ,1n}(\theta_1,i) \to \varrho_{1n}^{*}(\theta_1,i), 
\end{equation}
  for each $\theta_1 > 0$ and $i \in S$. From  Lemmata \ref{bd1} and  \ref{bd2}, we have $$ W(i)  \geq \bar{\psi}^{v_{2,n}^*}_{\alpha ,1n}(\theta_1, i) \, \geq \, \dfrac{1} {W(i_0) }\wedge \min_{i \in C_0^c} \{\bar{\psi}^{v_{2,n}^*}_{\alpha ,1n}(\theta_1, i) \}, \; 
i \in S.$$
Thus 
\interdisplaylinepenalty=0
\begin{eqnarray*}
 \lim_{\alpha \downarrow 0} \alpha \phi^{v_{2,n}^*}_{\alpha ,1n}(\theta_1,i)
 &=& \lim_{\alpha \downarrow 0}  \alpha(\frac{1}{\theta_1}\ln \bar{\psi}^{v_{2,n}^*}_{\alpha ,1n}(\theta_1, i)
   + \phi^{v_{2,n}^*}_{\alpha ,1n}(\theta_1,i_0) ),\\
    &=& \lim_{\alpha \downarrow 0}  \alpha \phi^{v_{2,n}^*}_{\alpha ,1n}(\theta_1,i_0) .
\end{eqnarray*}
Hence $\varrho_{1n}^{*}$ is a function of $\theta_1$ alone. Also by Lemma 
\ref{alpha1},  $\displaystyle{ \left\{
 \alpha \frac{d \phi^{v_{2,n}^*}_{\alpha ,1n}}{d \theta_1}(\theta_1,i) \Big| \alpha>0 \right\} }$ is bounded for each $i$. Hence along a further subsequence 
\begin{equation}\label{rho2}
 \alpha \frac{d \phi^{v_{2,n}^*}_{\alpha ,1n}}{d \theta_1}(\theta_1,i) \to \varrho_{2n}^*(\theta_1,i). 
\end{equation}
It follows from (\ref{rho1}) and (\ref{rho2}) that $\varrho_{2n}^{*}(\cdot,\cdot)=(\varrho_{1n}^{*})^{'}$ in the sense of distribution, where $(\varrho_{1n}^{*})^{'}$ is the distributional derivative (in $\theta_1$) of $\varrho_{1n}^{*}$. Hence $\varrho_{2n}^{*}(\cdot,\cdot)$ is also a function of $\theta_1$ alone.
Thus we have: for each $\theta_1 >0$, there exists a constant $\rho_{1n}^{*}$ such that along a suitable subsequence
$$
\theta_1 \left( \alpha \phi^{v_{2,n}^*}_{\alpha ,1n}(\theta_1,i)
+\theta_1 \alpha\dfrac{d \phi^{v_{2,n}^*}_{\alpha ,1n}}{d \theta_1}(\theta_1,i)\right) \to \rho_{1n}^{*}, \; \; i \in S.
$$
Using analogous arguments, we have along a suitable subsequence 
$$
\theta_2 \left( \alpha \phi^{v_{1,n}^*}_{\alpha ,2n}(\theta_2,i)
+\theta_2 \alpha\dfrac{d \phi^{v_{1,n}^*}_{\alpha ,2n}}{d \theta_2}(\theta_2,i)\right) \to \rho_{2n}^{*}, \; \; i \in S,
$$
where $\rho_{2n}^{*}$ is a constant.\\
Let all sequences above converge along a common subsequence $\alpha_m$ .
From Theorem \ref{thm17}, there exists a pair of Markov stationary strategies $(v_{1,n}^{\alpha_m},v_{2,n}^{\alpha_m})$ and a pair of bounded, absolutely continuous functions $(\bar{\psi}^{v_{2,n}^{\alpha_m}}_{\alpha_m,1n},\bar{\psi}^{v_{1,n}^{\alpha_m}}_{\alpha_m,2n})$  such that 
\interdisplaylinepenalty=0
 \begin{equation}\label{coupled-hjb17}
 \left\{\begin{aligned}
0&= \displaystyle{\inf_{v_1\in V_1}  \Big[ \Pi_{v_1,v_{2,n}^{\alpha_m}(i)} \bar{\psi}^{v_{2,n}^{\alpha_m}}_{\alpha_m,1n} +(\theta_1 r_1^n(i,v_1,v_{2,n}^{\alpha_m}(i))-g^{v_{2,n}^{\alpha_m}}_{\alpha_m,1n}(\theta_1,i))
\bar{\psi}^{v_{2,n}^{\alpha_m}}_{\alpha_m,1n}\Big ] } \\
&= \displaystyle{ \Big[ \Pi_{v_{1,n}^{\alpha_m},v_{2,n}^{\alpha_m}(i)} \bar{\psi}^{v_{2,n}^{\alpha_m}}_{\alpha_m,1n} +(\theta_1 r_1^n(i,v_{1,n}^{\alpha_m},v_{2,n}^{\alpha_m}(i))-g^{v_{2,n}^{\alpha_m}}_{\alpha_m,1n}(\theta_1,i))
\bar{\psi}^{v_{2,n}^{\alpha_m}}_{\alpha_m,1n}\Big ] } \\
\displaystyle{ \bar{\psi}^{v_{2,n}^{\alpha_m}}_{\alpha_m,1n}(\theta_1,i_0) } &= 1,\\
0&= \displaystyle{\inf_{v_2\in V_2}  \Big[ \Pi_{v_{1,n}^{\alpha_m}(i),v_{2}} \bar{\psi}^{v_{1,n}^{\alpha_m}}_{\alpha_m,2n} +(\theta_2 r_2^n(i,v_{1,n}^{\alpha_m}(i),v_{2})-g^{v_{1,n}^{\alpha_m}}_{\alpha_m,2n}(\theta_2,i))
\bar{\psi}^{v_{1,n}^{\alpha_m}}_{\alpha_m,2n}\Big ] } \\
&= \displaystyle{ \Big[ \Pi_{v_{1,n}^{\alpha_m}(i),v_{2,n}^{\alpha_m}(i)} \bar{\psi}^{v_{2,n}^{\alpha_m}}_{\alpha_m,2n} +(\theta_1 r_1^n(i,v_{1,n}^{\alpha_m}(i),v_{2,n}^{\alpha_m}(i))-g^{v_{1,n}^{\alpha_m}}_{\alpha_m,2n}(\theta_1,i))
\bar{\psi}^{v_{1,n}^{\alpha_m}}_{\alpha_m,2n}\Big ] } \\
\displaystyle{ \bar{\psi}^{v_{1,n}^{\alpha_m}}_{\alpha_m,2n}(\theta_1,i_0) } &= 1.
 \end{aligned}
 \right.
\end{equation}
Since $\mathcal S_1$ and $\mathcal S_2$ are compact, therefore along a subsequence denoted by the same subsequence  $\alpha_m$, we have $v_{1,n}^{\alpha_m} \to v_{1,n}^{*}$ and $v_{2,n}^{\alpha_m} \to v_{2,n}^{*}$ as $\alpha_m \to 0$, for some $(v_{1,n}^{*},v_{2,n}^{*}) \in \mathcal{S}_1 \times \mathcal{S}_2$.
The first part of the proof now follows by letting $\alpha_m \to 0$ in (\ref{coupled-hjb17}).
 
 Let $Y(t)$ be the process  corresponding to $(v_1,v_{2,n}^*)$ with 
initial condition $i \in S$.
Then using It$\hat{\rm o}$-Dynkin's formula and (\ref{brs17}), we get 
\begin{eqnarray*}
 E_i^{v_1,v^*_{2,n}} \Big [ e^{ \theta_1 \int_{0}^{T } (r_1^n(Y(s-),v_1(s,Y(s-)),v_{2,n}^*(Y(s-)))-\rho_{1n}^*) ds } \hat{\psi}_{1n}^*(Y(T ))\Big ]
 -\hat{\psi}_{1n}^*(i)&\geq& 0  . 
\end{eqnarray*}
Note that by Lemma \ref{ubd}, $\hat{\psi}_{1n}^*$ is bounded, which implies
\begin{eqnarray*}
 \hat{\psi}_{1n}^*(i) 
 &\leq& K(n)  E_i^{v_1,v^*_{2,n}} \Big [ e^{ \theta_1 \int_{0}^{T} (r_1^n(Y(s-),v_1(s,Y(s-)),v_{2,n}^*(Y(s-)))-\rho_{1n}^*) ds }
 \Big ] ,
\end{eqnarray*}
where $$K(n) \, = \, \max \Big\{ \max_{j=1, \dots, n} \hat{\psi}_{1n}^*(j), \  1 + 
\max_{j = 1, \cdots, n} W(j) \Big\}.$$
Taking logarithm, dividing by $\theta_1T$ and by letting $T \to \infty$, we obtain
\begin{eqnarray*}
 \rho_{1n}^* \leq \limsup_{T\rightarrow \infty}\dfrac{1}{\theta_1T}\ln  E_i^{v_1,v^*_{2,n}} \Big [ e^{ \theta_1 \int_{0}^{T} r_1^n(Y(s-),v_1(s,Y(s-)),v_{2,n}^*(Y(s-))) ds } \Big ]. 
\end{eqnarray*} 
Since $\theta_1 r_1^n \leq \theta_1 r_1\leq \delta $, it follows that $0 \leq \rho_{1n}^* \leq \delta $. Similarly we can show that $0 \leq \rho_{2n}^* \leq \delta $. Therefore we have $\displaystyle{\sup_{n}\{\rho^{*}_{1n},\rho^{*}_{2n} \}}\leq \delta$. This completes the proof.    
\end{proof}
Finally we prove that the coupled HJB equations described in Section 2 have suitable solutions which in turn leads to the existence of a Nash equilibrium in stationary strategies.
\begin{theorem}\label{bex18}
 Assume (A1), (A3) and (A4). Then there exist a pair of  stationary Markov strategies $(v_{1}^*,v_{2}^*) $, a pair of scalars   $(\rho^{*}_{1},\rho^{*}_{2})$ and  a pair of functions $(\hat{\psi}^{*}_{1}(i)), \hat{\psi}^{*}_{2}(i))$ in $B_{W}(S)\times B_{W}(S)$ such that 
 \interdisplaylinepenalty=0
  \begin{equation}\label{brs18}
 \left\{\begin{aligned}
 \theta_1 \rho^{*}_{1} ~\hat{\psi}^{*}_{1}(i) &= \inf_{v_1\in V_1} \Big [ \Pi_{v_1,v_{2}^*(i)} 
\hat\psi^{*}_{1}(i) +\theta_1 r_1(i,v_1,v_{2}^*(i))\hat\psi^{*}_{1}(i) \Big ]   \\
&=  \Pi_{v_{1}^*(i),v_{2}^*(i)} 
\hat\psi^{*}_{1}(i) +\theta_1 r_1(i,v_{1}^*(i),v_{2}^*(i))\hat\psi^{*}_{1}(i)   \\
\displaystyle{ \hat\psi^{*}_{1}(i_0) } &= 1, \\
 \theta_2 \rho^{*}_{2} ~\hat{\psi}^{*}_{2}(i) &= \inf_{v_2\in V_2} \Big [ \Pi_{v_{1}^*(i),v_2} 
\hat\psi^{*}_{2}(i) +\theta_2 r_2(i,v_{1}^*(i),v_{2})\hat\psi^{*}_{2}(i) \Big ]   \\
&=  \Pi_{v_{1}^*(i),v_{2}^*(i)} 
\hat\psi^{*}_{2}(i) +\theta_2 r_2(i,v_{1}^*(i),v_{2}^*(i))\hat\psi^{*}_{2}(i)  \\
\displaystyle{ \hat\psi^{*}_{2}(i_0) } &= 1. 
 \end{aligned}
 \right.
\end{equation}
Furthermore the pair of stationary Markov strategies $(v_{1}^*,v_{2}^*) $ is a Nash equilibrium and $(\rho^{*}_{1},\rho^{*}_{2})$ is the corresponding Nash values.
\end{theorem}
\begin{proof} By Theorem \ref{bex17}, we have a pair of stationary Markov strategies $(v_{1,n}^*,v_{2,n}^*) $, a pair of scalars   $(\rho^{*}_{1n},\rho^{*}_{2n})$ and  a pair of functions $(\hat{\psi}^{*}_{1n}(i)), \hat{\psi}^{*}_{2n}(i))$ in $B_{W}(S)\times B_{W}(S)$ satisfy the coupled HJB equations (\ref{brs17}).
Therefore by a diagonalization argument, 
along a suitable subsequence $\hat{\psi}_{1n}^*(i)\rightarrow \hat{\psi}_1^*(i)$ and $\hat{\psi}_{2n}^*(i)\rightarrow \hat{\psi}_2^*(i), \, i \in S$ for
$(\hat{\psi}_1^*, \hat{\psi}_2^* ) \in B_W(S)\times B_{W}(S)$. 

Since $\mathcal S_1$ and $\mathcal S_2$ are compact, therefore along a subsequence denoted by the same subsequence, we have $v_{1,n}^{*} \to v_{1}^{*}$ and $v_{2,n}^{*} \to v_{2}^{*}$ as $n \to \infty $, for  $(v_{1}^{*},v_{2}^{*}) \in \mathcal{S}_1 \times \mathcal{S}_2$.

Note that by Theorem \ref{bex17}, $\displaystyle{\sup_{n}\{\rho^{*}_{1n},\rho^{*}_{2n} \}}\leq \delta$, therefore along a subsequence denoted by the same subsequence, we have $\rho_{1n}^{*} \to \rho_{1}^{*}$ and $\rho_{2n}^{*} \to \rho_{2}^{*}$ as $n \to \infty $.
The first part of the proof now follows by letting $n \rightarrow \infty$  in (\ref{brs17}).

Let $Y(t)$ be the process  corresponding to $(v_1,v_2^*)$ with 
initial condition $i \in S$.
Then using (\ref{brs18}) and  It$\hat{\rm o}$-Dynkin's formula, see,  Guo and Hern$\acute{\rm a}$ndez-Lerma  [\cite{GuoLermabook}, Appendix C, pp. 218-219] (Note that the condition C.3 and C.4 satisfy under condition (A3) from [\cite{GuoLermabook}, Lemma 6.3, pp. 90-91]), we obtain 
\interdisplaylinepenalty=0
\begin{eqnarray*}
\hat{\psi}_1^*(i)  &\leq&  E_i^{ v_1,v_2^*} \Big [ e^{ \theta_1 \int_{0}^{T  } (r_1(Y(s-),v_1(s,Y(s-)),v_2^*(Y(s-)))-\rho_1^*) ds } \hat{\psi}_1^*(Y(T ))\Big ].
\end{eqnarray*}
It follows from Lemma \ref{bd1} that $\hat{\psi}_1^*(i)\leq W(i)$,  for all $i \in S$. Hence 
\begin{eqnarray*}
\hat{\psi}_1^*(i)  
&\leq& e^{-\theta_1 \rho_1^* T} E_i^{ v_1,v_2^*} \Big [ e^{ \theta_1 \int_{0}^{T} r_1(Y(s-),v_1(s,Y(s-)),v_2^*(Y(s-))) ds } W(Y(T ))\Big ].
\end{eqnarray*}
From Lemma {\ref{expbd3}}, we get
\begin{eqnarray*}
\hat{\psi}_1^*(i) &\leq& e^{- \theta_1 \rho_1^* T}  (W(i)+ b T) E_i^{ v_1,v_2^*} \Big [ e^{ \theta_1 \int_{0}^{T} r_1(Y(s-),v_1(s,Y(s-)),v_2^*(Y(s-))) ds }\Big ].
\end{eqnarray*}
Taking logarithm on both side we obtain
\begin{eqnarray*}
\ln \hat{\psi}_1^*(i)  &\leq& -\theta_1 \rho_1^* T + \ln (W(i)+ b T)+\ln  E_i^{ v_1,v_2^*} \Big [ e^{ \theta_1 \int_{0}^{T } r_1(Y(s-),v_1(s,Y(s-)),v_2^*(Y(s-))) ds } \Big ].
\end{eqnarray*}
 Now dividing by $\theta_1 T$ and by letting $T \to \infty$, we get
\begin{eqnarray*}
 \rho_1^* \leq \limsup_{T\rightarrow \infty}\dfrac{1}{\theta_1 T}\ln E_i^{ v_1,v_2^*} \Big [ e^{\theta_1  \int_{0}^{T} r_1(Y(s-),v_1(s,Y(s-)),v_2^*(Y(s-))) ds } \Big ],\; \; v_1 \in \mathcal{M}_1\,. 
\end{eqnarray*}
Let $v^*_1 \in \mathcal{S}_1$ be a minimizing selector in (\ref{brs18}) and let $Y(t)$ be the continuous time Markov chain corresponding to $(v^*_1,v_2^*) \in \mathcal{S}_1 \times \mathcal{S}_2 $ with initial condition $i$. Then using (\ref{brs18}) and It$\hat{\rm o}$-Dynkin's formula, we get 
\begin{eqnarray*}
E_i^{ v_1^*,v_2^*} \Big [ e^{ \theta_1 \int_{0}^{T } (r_1(Y(s-),v_1^*(Y(s-)),v_2^*(Y(s-)))-\rho_1^*) ds } \hat{\psi}_1^*(Y(T ))\Big ]
 -\hat{\psi}_1^*(i)&=& 0  . 
\end{eqnarray*}
Since $\hat{\psi}$ is bounded below by Lemma \ref{bd2}, we have
\begin{eqnarray*}
 \hat{\psi}_1^*(i) 
 &\geq& K E_i^{ v_1^*,v_2^*} \Big [ e^{ \theta_1 \int_{0}^{T} (r_1(Y(s-),v_1^*(Y(s-)),v_2^*(Y(s-)))-\rho_1^*) ds }
 \Big ],
\end{eqnarray*}
where $$ K  = \dfrac{1} {W(i_0) }\wedge \min_{i \in C_0^c} \{\hat{\psi}^*_1(i) \}.$$
Taking logarithm, dividing by $\theta_1  T$ and by letting $T \to \infty$, we get
\begin{eqnarray*}
 \rho_1^* \geq \limsup_{T\rightarrow \infty}\dfrac{1}{\theta_1 T}\ln E_i^{ v_1^*,v_2^*} \Big [ e^{ \theta_1 \int_{0}^{T} r_1(Y(s-),v_1^*(Y(s-)),v_2^*(Y(s-))) ds } \Big ]. 
\end{eqnarray*}
 Therefore 
\begin{equation*}
\rho_1^* =\rho_1^{v_1^*,v_2^*}\leq \rho_1^{v_1,v_2^*} \; \; \forall v_1 \in \mathcal{M}_1.
\end{equation*} 
Using analogous argument we can show that
\begin{equation*}
\rho_2^* =\rho_2^{v_1^*,v_2^*}\leq \rho_2^{v_1^*,v_2} \; \; \forall v_2 \in \mathcal{M}_2.
\end{equation*} 
This completes the proof.   
\end{proof}
\begin{remark}
 Note that $\rho_i^*, \; v_i^*, \; i=1,2$,  depend on $\theta_1,\; \theta_2$. As before we have suppressed this dependence for notational convenience.
\end{remark}

\section{Conclusion} We have established the existence of a pair of stationary strategies which constitutes a pair of Nash equilibrium strategies for risk sensitive stochastic games for continuous time Markov chain with ergodic cost. We have achieved these under a Lyapunov type stability assumption (A3) and a small cost condition (A4) which lead to the existence of suitable solution to the corresponding coupled HJB equations. The Lyapunov type of stability assumption is standard in literature (see, e.g., Guo and Hern$\acute{\rm a}$ndez-Lerma  \cite{GuoLermabook}). The small cost assumption mean that the risk aversion parameter $\theta_k$ of player $k$ must be small. For the discounted cost criterion we have established the existence of Nash equilibrium in Markov strategies under an additive structure (A2). It will be interesting to investigate if (A2) can be dropped to achieve the same result.
\section*{Acknowledgments.} The work of the first named author is supported in part by UGC Centre for Advanced Study. The work of the second named author is supported in part by the DST, India project no. SR/S4/MS:751/12. The work of the third named author is supported in part by Dr. D. S. Kothari postdoctoral fellowship of UGC.

%\section{Introduction.}\label{intro} %%1.
%\subsection{Duality and the classical EOQ problem.}\label{class-EOQ} %% 1.1.
%\subsection{Outline.}\label{outline1} %% 1.2.
%\subsubsection{Cyclic schedules for the general deterministic SMDP.}
%  \label{cyclic-schedules} %% 1.2.1
%\section{Problem description.}\label{problemdescription} %% 2.

% Text of your paper here

% Appendix here
% Options are (1) APPENDIX (with or without general title) or 
%             (2) APPENDICES (if it has more than one unrelated sections)
% Outcomment the appropriate case if necessary
%
% \begin{APPENDIX}{<Title of the Appendix>}
% \end{APPENDIX}
%
%   or 
%
% \begin{APPENDICES}
% \section{<Title of Section A>}
% \section{<Title of Section B>}
% etc
% \end{APPENDICES}

% Acknowledgments here

% Enter the text of acknowledgments here

% References here (outcomment the appropriate case) 

% CASE 1: BiBTeX used to constantly update the references 
%   (while the paper is being written).
%\bibliographystyle{ormsv080} % outcomment this and next line in Case 1
%\bibliography{<your bib file(s)>} % if more than one, comma separated

% CASE 2: BiBTeX used to generate mypaper.bbl (to be further fine tuned)
%\input{mypaper.bbl} % outcomment this line in Case 2

\end{document}